\documentclass[reqno,11pt]{amsart}

\usepackage{tikz}
\usetikzlibrary{shapes, arrows.meta, positioning, fit}
\usepackage{amsthm,amsmath,amssymb}
\usepackage{mathrsfs,amsfonts,dsfont,functan,extarrows,mathtools}

\usepackage{hyperref}
\hypersetup{colorlinks=true,linkcolor=black, urlcolor=blue, citecolor=blue}

\usepackage[margin=2.8cm, a4paper]{geometry}

\usepackage{ulem}
\usepackage{marginnote}
\usepackage{xcolor}

\newtheorem{theorem}{Theorem}[section]

\newtheorem{definition}[theorem]{Definition}
\newtheorem{corollary}[theorem]{Corollary}

\newtheorem{lemma}[theorem]{Lemma}

\numberwithin{equation}{section}
\allowdisplaybreaks
\def\d{\mathrm{d}}
\def\div{{\rm div}}
\newcommand{\la}{\langle}
\newcommand{\ra}{\rangle}
\newcommand{\na}{\nabla}

\newcommand{\fab}{^{\frac{1}{2}}}

\newcommand{\no}{\nonumber}

\begin{document}

\title[Convergence from 2-NSM to NSM-Ohm]{On the convergence to the Navier-Stokes-Maxwell system with solenoidal Ohm's law}
\author[Z. Guo]{Zihua Guo}
\address[Zihua Guo]{School of Mathematics, Monash University, Clayton VIC 3800, Australia}
\email{zihua.guo@monash.edu}
\author[Z. Zhang]{Zeng Zhang}
\address[Zeng Zhang]{
School of Mathematics and Statistics, Wuhan University of Technology, Wuhan, 430070, P. R. China}
\email{zhangzeng534534@whut.edu.cn}

\date{}

\begin{abstract}
The incompressible Navier-Stokes-Maxwell system with solenoidal Ohm’s law can be viewed as  as the asymptotic limit of the two-fluid incompressible Navier-Stokes-Maxwell system as the momentum transfer coefficient tends to zero (see \cite{AIM2015}, Ars\'{e}nio, Ibrahim and Masmoudi, Arch. Ration. Mech. Anal., 2015). We prove this limit rigorously without loss of regularity by using the idea of frequency envelope. 

\vspace*{5pt}

\end{abstract}

\maketitle

\section{Introduction} 
\label{sec:intro}
\subsection{The models}
We consider  the following two fluid incompressible Navier-Stokes-Maxwell system:
\begin{align*}\label{NSM}\tag{NSM}
	  \left\{
	\begin{array}{ll}
	\partial_tu^++u^+\cdot \nabla u^+-\mu\Delta u^+ +\frac{1}{2\sigma \varepsilon^2}(u^+-u^-)=-\nabla p^++\frac{1}{\varepsilon}(cE+u^+\times B), \\[1ex]
	\partial_tu^-+u^-\cdot \nabla u^--\mu\Delta u^- -\frac{1}{2\sigma \varepsilon^2}(u^+-u^-)=-\nabla p^--\frac{1}{\varepsilon}(cE+u^-\times B), \\ [1ex]
	\frac{1}{c}\partial_tE-\nabla \times B=-\frac{1}{2 \varepsilon}(u^+-u^-),\\[1ex]
	\frac{1}{c}\partial_tB+\nabla \times E=0,\\[1ex]
	\div u^+=0,\ \div u^-=0,\ \div E=0,\ \div B=0,
	\end{array}
	\right.
	\end{align*}
with the initial data
	\begin{align*}
	  &u^+(t,x)|_{t=0}=u^{+,in}(x),\  u^-(t,x)|_{t=0}=u^{-,in}(x),\\
	  &E(t,x)|_{t=0}=E^{in}(x),\ B(t,x)|_{t=0}=B^{in}(x).
	\end{align*}
This system models the motion of a  plasma of oppositely charged ions.   Here, the vector fields $u^+, u^-:\mathbb{R}^+_t\times \mathbb{R}^d_x\rightarrow \mathbb{R}^3$ represent the velocities of the cations and anions, respectively, with $d=2$ or $3$. The scalar functions $p^+$ and $p^-:\mathbb{R}^+_t\times \mathbb{R}^d_x\rightarrow \mathbb{R}$
stand for the pressure. The third equation is the Amp\`{e}re-Maxwell equation  for  the electric field $E:\mathbb{R}^+_t\times \mathbb{R}^d_x\rightarrow \mathbb{R}^3$, and the fourth equation is the Faraday’s law for the magnetic field $B:\mathbb{R}^+_t\times \mathbb{R}^d_x\rightarrow \mathbb{R}^3$.  The constant $\mu>0$ represents the viscosity,  $\sigma > 0$ denotes the
electrical conductivity of the fluid,  $c>0$ is the speed of light, and $\varepsilon>0$ is the momentum transfer coefficient. For the physical background of \eqref{NSM}, we refer to 
\cite{AIM2015,B1933,
D2001,GY1984}. We also mention that a rigorous derivation of  \eqref{NSM}  as the hydrodynamic limit from kinetic theory is established in \cite{AS2019}.

Multiplying  \eqref{NSM} by $(u^+,u^-,E,B)$, assuming they decay
 at infinity, then integrating in space and  using the divergence-free conditions,  one can obtain  the
following  energy conservation law:
\begin{align*}
  \tfrac{1}{2}\tfrac{\d}{\d t}(\tfrac{1}{2}\|u^+\|_{L^2}^2+&\tfrac{1}{2}\|u^-\|_{L^2}^2+\|E\|_{L^2}^2+\|B\|_{L^2}^2)\\&+
  \tfrac{\mu}{2}(\|\nabla u^+\|_{L^2}^2+\|\nabla u^-\|_{L^2}^2)+\tfrac{1}{\sigma}\|\tfrac{u^+-u^-}{2 \varepsilon}\|_{L^2}^2=0.
\end{align*}
Following the method of Leray \cite{L1934}, Giga, Ibrahim, Shen and Yoneda \cite{GISY2014}  established the  existence of global weak solutions
for any initial data $(u^{+,in},u^{-,in},E^{in},B^{in})\in (L^2)^4$ in both 2D and 3D. Using the fixed point argument,  the authors \cite{GISY2014} also obtained local mild solutions for large data  and global solutions for small data for $(u^{+,in},u^{-,in},E^{in},B^{in})\in H^{\frac{1}{2}}(\mathbb{R}^3)\times H^{\frac{1}{2}}(\mathbb{R}^3)\times L^2(\mathbb{R}^3)\times L^2(\mathbb{R}^3)$.

As $c$ tends to infinity and $\varepsilon$ tends to zero, there are three distinct asymptotic regimes:
\begin{itemize}
    \item fixing $\varepsilon$ and letting $c$ tends to infinity, then  letting $\varepsilon$ tends to zero;
    
      \item fixing $c$ and letting $\varepsilon$ tends to zero, then  letting $c$ tends to infinity;
    
    \item letting $c$ tends to infinity and  $\varepsilon$ tends to zero simultaneously. 
  
\end{itemize}
As mentioned in \cite{AIM2015}, the following relationship diagram (Figure 1) can be obtained formally.

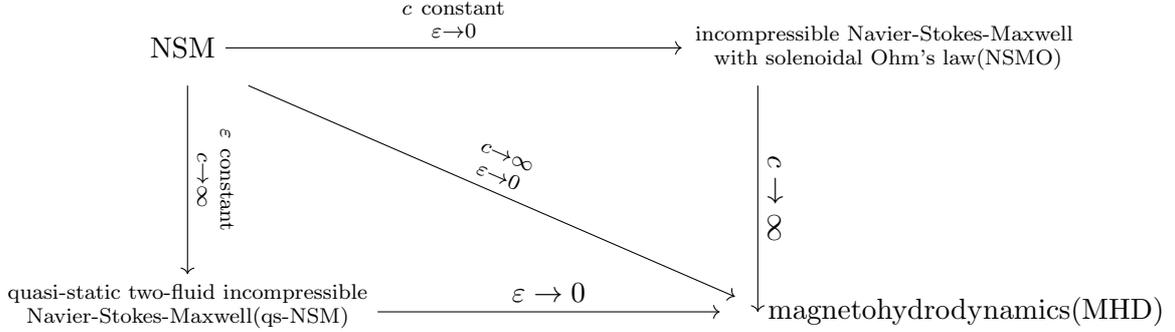
\begin{figure}[h]
\centering
\begin{tikzpicture}
        \draw[->] (0,0) node[left]{NSM}-- node [pos=0.5,above,sloped]{${c ~~\text{constant}}\atop{\varepsilon \rightarrow 0}$}(6,0) node[right] { 
         ${\text{incompressible Navier-Stokes-Maxwell }}
         \atop 
         {\text{with solenoidal Ohm's law(NSMO)}}$
         };;
         
        \draw[->] (7,-0.5) -- node [above,sloped]{$c\rightarrow \infty$}(7,-3.5) node[right]{magnetohydrodynamics(MHD)};

        \draw[->] (-0.5,-0.5)-- node[above,sloped,]
         {$
         {\varepsilon ~~\text{constant}}
         \atop
         {c \rightarrow \infty}
         $}
         (-0.5,-3) node[below]
         { 
         ${\text{quasi-static two-fluid incompressible}}
         \atop 
         {\text{Navier-Stokes-Maxwell(qs-NSM)}}$
         };

         \draw[->] (2,-3.5)-- node[above,sloped,]
         {$
         \varepsilon\rightarrow 0
         $}
         (6.5,-3.5) ;

           \draw[->] (0.3,-0.5)-- node[above,sloped,]
         {${c \rightarrow \infty}\atop{\varepsilon \rightarrow 0}$}
         (6.7,-3.3) ;
        
    \end{tikzpicture}
\caption{asymptotic regimes}
\end{figure}
To keep the text more coherent and concise, we focus only on the system (NSMO). For further details on the systems (qs-NSM) and (MHD), we refer to \cite{AIM2015}.
First, denoting the bulk velocity $u=\frac{1}{2 }(u^++u^-)$ and the electrical current $j=\frac{1}{2 \varepsilon}(u^+-u^-)$,  we rewrite \eqref{NSM} in the following equivalent form:
	\begin{align}\label{euqNSM}\tag{eq-NSM}
	  \left\{
		\begin{array}{l}
			\partial_tu+u\cdot \nabla u+\varepsilon^2j\cdot \nabla j-\mu\Delta u=-\nabla p+j\times B,
		\\[1ex]
			\varepsilon^2\partial_tj+\varepsilon^2(u\cdot \nabla j+j\cdot \nabla u)-\varepsilon^2\mu\Delta j+\frac{1}{\sigma}j=-\nabla \tilde{p}+cE+u\times B,
		\\[1ex]
			\frac{1}{c}\partial_tE-\nabla \times B=-j,
		\\[1ex]
			\frac{1}{c}\partial_tB+\nabla \times E=0,
		\\[1ex]
			\div u=0,\ \div j=0,\ \div E=0,\ \div B=0,
		\end{array}
		\right.
	\end{align}
where the pressures $p=\tfrac{1}{2}(p^++p^-)$ and $\tilde{p}=\tfrac{\varepsilon}{2}(p^+-p^-)$.
Then letting $\varepsilon\rightarrow 0$, 
one can obtain formally the following incompressible Navier-Stokes-Maxwell system with solenoidal Ohm's law:
	\begin{align}\tag{NSMO}\label{NSMO}
		 \left\{
		\begin{array}{ll}
			\partial_tu+u\cdot \nabla u-\mu\Delta u=-\nabla p+j\times B,  \\[1ex]
			j=\sigma(-\nabla \tilde{p}+cE+u\times B),   \\[1ex]
			\frac{1}{c}\partial_tE-\nabla \times B=-j,  \\[1ex]
			\frac{1}{c}\partial_tB+\nabla \times E=0, \\[1ex]
			\div u=0,\ \div j=0,\ \div E=0,\ \div B=0.
		\end{array}
		\right.
	\end{align}
It can be verified that \eqref{NSMO} enjoys the conservation law:
 \begin{align*}
  \tfrac{1}{2}\tfrac{\d}{\d t}(\|u\|_{L^2}^2+\|E\|_{L^2}^2+\|B\|_{L^2}^2)+
  \mu\|\nabla u\|_{L^2}^2+\tfrac{1}{\sigma}\|j\|_{L^2}^2=0.
\end{align*}
However, proving the global existence of Leray-weak solutions to \eqref{NSMO} is challenging. This difficulty arises from the hyperbolic nature of the Maxwell equations, which makes it hard to establish the weak stability of the nonlinear term $j\times B$ in $L^2$-energy space. For analytical studies of \eqref{NSMO}, we refer  to \cite{GIM2014,M2010} and the references therein. We mention that the results in \cite{GIM2014,M2010} pertain to another system, which is mathematically very much similar
to \eqref{NSMO} and is also called an incompressible Navier–Stokes–Maxwell
system with solenoidal Ohm’s law in some literature.  The method developed
in \cite{GIM2014,M2010} can be directly applied to analyze and solve \eqref{NSMO}. Using the energy inequality and a logarithmic estimate to bound the $L^\infty_x$-norm of the velocity field,  Masmoudi \cite{M2010} established the existence and uniqueness of global solution for $(u^{in},E^{in},B^{in})\in L^2(\mathbb{R}^2)\times (H^{s}(\mathbb{R}^2))^2$ with $0<s<1$.  Greman, Ibrahim and Masmoudi \cite{GIM2014} obtained the local existence of mild solutions for large data and global existence for small data, for $(u^{in},E^{in},B^{in})\in L^2\times (L^2_{log})^2$ in 2D and $(u^{in},E^{in},B^{in})\in (\dot{H}^{\frac{1}{2}})^{3}$ in 3D.  Their proof relies on an $L^2_tL^\infty_x$ estimates for the velocity field. For classical solutions,  Jiang and Luo \cite{JL2018}  employed  the
decay and dissipative properties of the electromagnetic field and established the global in time  solution for small initial data $u^{in}\in H^s(\mathbb{R}^3),\ E^{in},\ B^{in}\in H^{s+1}(\mathbb{R}^3)$ with $s\geq 2$.

\subsection{Rigorous justifications}
The limits 
\begin{align*}
\eqref{NSM}\xrightarrow[]{\substack
{\varepsilon ~~\text{constant}}
\atop
{c\rightarrow \infty} 
}\text{(qs-NSM)}
\xrightarrow []{c\rightarrow\infty} \text{(MHD)}
\end{align*}
has been rigorously justified  in \cite{AIM2015} within the framework of Leray-weak solutions, using the natural uniform energy bounds and the classical Aubin-Lions-Simon compactness theorem (see, for instance, \cite{s1987}).  However, the limit
\begin{align*}
\eqref{NSM}\xrightarrow[]{\substack
{c ~~\text{constant}}
\atop
{\varepsilon\rightarrow 0} 
}\text{(NSMO)}
\end{align*} 
in the framework of Leray weak solution remains an open problem, as well as the the existence of  global weak solutions to NSMO. The primary difficulty lies in handling the nonlinear term $j^\varepsilon\times B^\varepsilon$ when passing to the limit. This highlights the  singularity of this regime compared to the first one.  As for the limit
\begin{align*}
\eqref{NSM}\xrightarrow[]{\substack
{c \rightarrow \infty}
\atop
{\varepsilon\rightarrow 0} 
}\text{(MHD)}, 
\end{align*} 
under the assumption that $\varepsilon$ is not too small relative to $c$, Ars{\'e}nio, Ibrahim and Masmoudi
 \cite{AIM2015} established a rigorous convergence result for weak solutions. This assumption aims at staying away from the singular system NSMO. 
The approach is based on   the spectral analysis of Maxwell's operator, which play a key role  in establishing weak stability of nonlinear terms. We mention that using this spectral analysis, the  convergence
$
{\text{(NSMO)}}\xrightarrow[]{
c \rightarrow \infty
}\text{(MHD)}
$
is also rigorously justified in \cite{AIM2015}.

In this paper, we focus on the convergence of \eqref{NSM} to (NMSO). By imposing more regularity on the initial data, it becomes feasible to establish  a rigorous asymptotic result. The second author of the current paper \cite{Z2021} proved that \eqref{NSM} admits a unique global solution for small initial data
\begin{align*}
u^{+,\varepsilon,in},u^{-,\varepsilon,in},E^{\varepsilon,in},B^{\varepsilon,in}
\in H^s(\mathbb{R}^3)\,\,
\text{with}\,\,  s\geq 3.
\end{align*}
Furthermore, if the initial data $$
(\tfrac{u^{+,\varepsilon,in}+u^{-,\varepsilon,in}}{2},u^{+,\varepsilon,in}-u^{-,\varepsilon,in},
E^{\varepsilon,in},B^{\varepsilon,in})$$ converges strongly to $(u^{in}, 0,
E^{in},B^{in})$ in a suitable  sense as $\varepsilon\rightarrow0$, then
$(\tfrac{u^{+,\varepsilon}+u^{-,\varepsilon}}{2},
E^{\varepsilon},B^{\varepsilon})$ is a Cauchy sequence in  $(L^\infty(\mathbb{R}^+;H^{s-2}(\mathbb{R}^3)))^3$, and
$
\tfrac{u^{+,\varepsilon}-u^{-,\varepsilon}}{2\varepsilon}$ is a Cauchy sequence in  $L^2(\mathbb{R}^+;H^{s-2}(\mathbb{R}^3))$. The limit functions $(u, E, B, j)$ then satisfy \eqref{NSMO} with the initial data $(u^{in}, E^{in}, B^{in}).$

It is worth mentioning that in the results of \cite{Z2021}, the initial data belong to $H^s$, while the solution converges in $H^{s-2}$. Clearly, the regularity of the latter is lower than that of the former. We aim to improve these results by demonstrating that the
regularity of the space in which the solutions converge can be the same as that of the initial data.  
\subsection{Statement of the results}Now we state our main results as follows:
\begin{theorem}[Global well-posedness for \eqref{NSM}] \label{thm:GNSM}
  Let $s> \tfrac{3}{2}$ and $s-1\leq s'\leq s+1$. Assume the initial data $u^{+,\varepsilon,in}, u^{-,\varepsilon,in},  E^{\varepsilon,in}, B^{\varepsilon,in}$ are all divergence free and satisfy
  \begin{align*}
  u^{+,\varepsilon,in}+u^{+,\varepsilon,in}\in H^{s'}(\mathbb{R}^3),\,
  u^{+,\varepsilon,in}-u^{-,\varepsilon,in}, E^{\varepsilon,in}, B^{\varepsilon,in}\in H^{s}(\mathbb{R}^3).
  \end{align*}
 There exists a small positive constant $\kappa_0$, independent of $\varepsilon$,  such that if
	  \begin{align*}	  \|u^{+,\varepsilon,in}+u^{-,\varepsilon,in}\|_{H^{s'}(\mathbb{R}^3)}^2 + \| u^{+,\varepsilon,in}-u^{-,\varepsilon,in}\|_{H^s(\mathbb{R}^3)}^2
		  + \|E^{\varepsilon,in}\|_{H^{s}(\mathbb{R}^3)}^2
		  + \|B^{\varepsilon,in}\|_{H^{s}(\mathbb{R}^3)}^2
			\leq \kappa_0,
	  \end{align*}
then the system \eqref{NSM} admits a unique global in time  solution $(u^{+,\varepsilon}, u^{-,\varepsilon},$ $E^{\varepsilon}, B^{\varepsilon})$ satisfying
	  \begin{align*}
 &u^{+,\varepsilon}+u^{-,\varepsilon} \in C(\mathbb{R}^+,{H}^{s'}(\mathbb{R}^3)),\,\,
 \nabla(u^{+,\varepsilon}+u^{-,\varepsilon})\in L^2(\mathbb{R}^+;{H}^{s'}(\mathbb{R}^3)),\\
	  & u^{+,\varepsilon}-u^{-,\varepsilon} \in C(\mathbb{R}^+,{H}^{s}(\mathbb{R}^3)),\,\,
 \nabla(u^{+,\varepsilon}-u^{-,\varepsilon})\in L^2(\mathbb{R}^+;{H}^{s}(\mathbb{R}^3)),\\
	  & E^\varepsilon,\,\, B^\varepsilon\in C(\mathbb{R}^+;H^s(\mathbb{R}^3)).
	  \end{align*}

\end{theorem}

\begin{theorem}[Strong convergence from  \eqref{NSM} to \eqref{NSMO}] \label{thm:Global-converg}
   Let $s> \tfrac{3}{2}$ and $s-1\leq s'\leq s+1$.
 Assume that for all $\varepsilon>0$,
    \begin{align*}	  \|u^{+,\varepsilon,in}+u^{-,\varepsilon,in}\|_{H^{s'}(\mathbb{R}^3)}^2 + \| u^{+,\varepsilon,in}-u^{-,\varepsilon,in}\|_{H^s(\mathbb{R}^3)}^2
		  + \|E^{\varepsilon,in}\|_{H^{s}(\mathbb{R}^3)}^2
		  + \|B^{\varepsilon,in}\|_{H^{s}(\mathbb{R}^3)}^2
			\leq \kappa_0
	  \end{align*}
with $\kappa_0$ a small positive constant. We further assume that
\begin{align*}
 (\tfrac{u^{+,\varepsilon,in} + u^{-.\varepsilon,in}}{2},u^{+,\varepsilon,in} - u^{-,\varepsilon,in}, E^{\varepsilon,in},
  B^{\varepsilon,in}) \rightarrow (u^{in},0, E^{in}, B^{in}) \,\,\, \textit{in}\,\,\, H^{s'}(\mathbb{R}^3)\times (H^{s}(\mathbb{R}^3))^3,
  \end{align*}
as  $\varepsilon\rightarrow 0.$ Let $\{(u^{+,\varepsilon}, u^{-,\varepsilon}, E^{\varepsilon}, B^{\varepsilon})\}$ be the family of global solutions to  \eqref{NSM} constructed in Theorem \ref{thm:GNSM}. Then there are functions
$(u(t,x),j(t,x),E(t,x),B(t,x))$ satisfying, as  $\varepsilon\rightarrow 0,$
\begin{align*}
		 & (\tfrac{u^{+,\varepsilon} + u^{-,\varepsilon}}{2}, E^\varepsilon,B^\varepsilon)  \rightarrow (u,E,B)\,\,\text{in}\,\,
 L^\infty(\mathbb{R}^+;H^{s'}(\mathbb{R}^3))\times (L^\infty(\mathbb{R}^+;H^{s}(\mathbb{R}^3))^2,\\
		 & ( \tfrac{\na(u^{+,\varepsilon} + u^{-,\varepsilon})}{2}, \tfrac{u^{+,\varepsilon} - u^{-,\varepsilon}}{2\varepsilon})  \rightarrow (\na u, j)\,\,\text{in}\,\,
 L^2(\mathbb{R}^+;H^{s'}(\mathbb{R}^3))\times L^2(\mathbb{R}^+;H^{s}(\mathbb{R}^3)).
	  \end{align*}
Moreover, the limit functions $(u,j,E,B)$ subject to the system \eqref{NSMO} with the initial data
$$u|_{t=0}=u^{in},E|_{t=0}=E^{in},B|_{t=0}=B^{in}.$$

\end{theorem}

  Compared to the previous work \cite{Z2021}, the regularity of the initial velocity and electromagnetic field has been reduced. This improvement is due to the use of estimates in Besov spaces, which are more refined than those in integer Sobolev spaces. We mention that the values of $s'$ and $s$ may not be optimal, but this is not the main focus of the paper. The main advancement in our work lies in eliminating the loss of regularity in the convergence space. Specifically,  the initial data belong to $H^{s'}\times (H^s)^3$, and the solutions still converge in $H^{s'}\times (H^s)^3$. Our strategy involves decomposing $(u^{\varepsilon}, j^{\varepsilon}, E^{\varepsilon}, B^{\varepsilon})$ (recall that $u^\varepsilon=\tfrac{1}{2 }(u^{+,\varepsilon}+u^{-,\varepsilon})$ and $j^\varepsilon=\tfrac{1}{2 \varepsilon}(u^{+,\varepsilon}-u^{-,\varepsilon})$) into low and high frequencies. \begin{itemize}
 \item After arbitrarily extracting a sequence $\{\varepsilon_n\}$ with $\lim\limits_{n\rightarrow \infty}\varepsilon_n=0$, and by introducing weighted Besov spaces (see Preliminary 2.1), we can show that there is a uniform small tail estimate for solutions $(u^{\varepsilon_n}, j^{\varepsilon_n}, E^{\varepsilon_n}, B^{\varepsilon_n})$ (see the estimate \eqref{hign-diff}). More precisely, we obtain that
    $\forall \eta>0$, there exists a $k_0$, such that for all $n$,
\begin{align*}
\|u^{\varepsilon_n}\|_{L^\infty(\mathbb{R}^+;H^{s'}_{k> k_0})}^2&+
\|\na u^{\varepsilon_n}\|_{L^2(\mathbb{R}^+;{H}^{s'}_{k >k_0})}^2\\&+
\|j^{\varepsilon_n}\|_{L^2(\mathbb{R}^+;{H}^{s}_{k >k_0})}^2
+\|E^{\varepsilon_n},B^{\varepsilon_n}\|_{L^\infty(\mathbb{R}^+;H^{s}_{k> k_0})}^2\leq \eta.
\end{align*}
Here $\|f\|_{H^{l}_{>k_0}}=\Big ( \sum\limits_{k>k_0} 2^{2k{l}} \|{\Delta}_k f\|^2_{L^2}   \Big)^{\frac{1}{2}}$ and $\|f\|_{H^{l}_{\leq k_0}}=\Big ( \sum\limits_{k\leq k_0} 2^{2k{l}} \|{\Delta}_k f\|^2_{L^2}   \Big)^{\frac{1}{2}}$.
\item For low frequency,  we have the following useful embedding inequality: \begin{align*}
\|f\|_{H^{l+l'}_{\leq k_0}}\lesssim
   2^{k_0l'} \| f \|_{H^{l}_{\leq k_0}}~
\text{ for all}~ l'>0~(\text{see~ the~inequality~ \eqref{l+}}).
\end{align*}

A critical step in the convergence analysis is to show $\varepsilon^2\partial_t j^\varepsilon$ tends to $0$ in some sense as $\varepsilon\rightarrow 0$. Back to \eqref{euqNSM}-(2). Due to the energy estimate  $\varepsilon\nabla j^\varepsilon\in L^2(\mathbb{R}^+,H^{s})$ (see \eqref {uni-2}), we observe that $\varepsilon^2\mu\Delta j^\varepsilon\in L^2(\mathbb{R}^+,H^{s-1})$. Therefore, 
we naturally avoid working in the more restrictive $L^2(\mathbb{R}^+,H^{s})$ space for $\varepsilon^2\partial_t j^\varepsilon.$  In fact, we  obtain that $\varepsilon^2\partial_t j^\varepsilon\in L^2({\mathbb{R}^+;H^{s-1}})$ and $\|\varepsilon^2\partial_tj^\varepsilon\|_{L^2({\mathbb{R}^+;H^{s-2}})}\rightarrow 0 $ as $\varepsilon\rightarrow 0$ in \cite{Z2021}.

However, after performing a truncation, we can prove that $\varepsilon^2\partial_t j^\varepsilon\in L^2({\mathbb{R}^+;H^{s}_{\leq k_0}})$ and \begin{align*}
   \|\varepsilon^2\partial_t j^\varepsilon\|_{L^2(\mathbb{R}^+;H^{s}_{\leq k_0})}^2\lesssim \|\varepsilon j^{\varepsilon,in}\|_{H^{s}}^2+\varepsilon2^{2k_0}
 \text{(see~ Lemma~ \ref{hf})}.
  \end{align*}
The appearance of $2^{2k_0}$
on the right-hand side is natural due to the embedding inequality. It is crucial that there is also $\varepsilon$ appears in front of $2^{2k_0}$. This indicates that for fixed $k_0$, $\varepsilon2^{2k_0}$ can be arbitrarily small if $\varepsilon$ is sufficiently small.

Taking the energy estimates for low frequencies, and applying Lemma \ref{hf},  we can prove (see Lemma \ref{low-diff}) that there exists a positives number $N_0(\eta,k_0)$, if $n_1,n_2\geq N_0,$ then we have
\begin{align*}
\|&u^{\varepsilon_{n_2}}-u^{\varepsilon_{n_1}}\|_{L^\infty(\mathbb{R}^+;H^{s'}_{k\leq k_0})}^2+
\|\na( u^{\varepsilon_{n_2}}-u^{\varepsilon_{n_1}})\|_{L^2(\mathbb{R}^+;{H}^{s'}_{k \leq k_0})}^2\\&+
\|j^{\varepsilon_{n_2}}-j^{\varepsilon_{n_1}}\|_{L^2(\mathbb{R}^+;{H}^{s}_{k \leq k_0})}^2
+\|E^{\varepsilon_{n_2}}-E^{\varepsilon_{n_1}},B^{\varepsilon_{n_2}}-
B^{\varepsilon_{n_1}}\|_{L^\infty(\mathbb{R}^+;H^{s}_{k\leq  k_0})}^2\lesssim \eta.
\end{align*}
  \end{itemize}
 Our strategy is inspired by the method of frequency envelope. This method has been widely used to prove the continuous dependence on the initial date (see, for instance, \cite{GYZ2024, KT2003,T2004}), as well as to address the asymptotic limit problems (see, for instance \cite{MN2008, MN2002}). We believe that this method can also be applied to other limit problems.

The organization of this paper is as follows.
Section 2 is devoted to establishing some preliminary lemmas.
In Section 3, we establish  a priori energy estimates, and then we obtain the global existence of solutions to \eqref{NSM} for small initial data.
In Section 4, we show
$\varepsilon^2\partial_t j^\varepsilon\in L^2({\mathbb{R}^+;H^{s}_{\leq k_0}})$ by using the decay of the electrical field and the dissipation of the magnetic field.
In Section 5, we prove that by extracting an arbitrary  sequence,  $(u^{\varepsilon},\na u^{\varepsilon},j^{\varepsilon},E^{\varepsilon},B^{\varepsilon})$ is a Cauchy sequence in $L^\infty(\mathbb{R}^+;H^{s'})\times L^2(\mathbb{R}^+;{H}^{s'})\times L^2(\mathbb{R}^+;H^{s})\times (L^\infty(\mathbb{R}^+;H^{s}))^2,$ and  the limit functions solve the system \eqref{NSMO}.

We use $A\lesssim B$ to denote the inequality $A\leq CB$, where $C$ is a positive constant  independent of $\varepsilon$. Additionally, we  use $A\thickapprox B$ to represent $C_1A\leq B\leq C_2$, where $C_1$ and $C_2$  are  positive constants, also independent of $\varepsilon$. 
\section{Preliminaries}
\subsection{Besov space and weighted Besov spaces}
Let $(\chi,\varphi)$ be a couple of $C^{\infty}$ functions with
$
\text{Supp}\,\chi \subset \{|\xi|\leq 4/3\},$
$\text{Supp}\,\varphi\subset \{3/4\leq|\xi|\leq 8/3\},$ and
$\chi(\xi)+\sum\limits_{k\geq0}\varphi(\xi/2^k)=1, \forall\,\, \xi.$
Denote $\tilde{h}=\mathcal{F}^{-1}\chi$ and $h=\mathcal{F}^{-1}\varphi$, where $\mathcal{F}^{-1}$ is the inverse Fourier transform.
We define  the dyadic blocks as follows.
\begin{align*}
&\Delta_k u \overset{\text{def}}{=}0,\,\,\text{if}\,\,k\leq -2,\,\,\,
 \Delta_{-1} u\overset{\text{def}}{=}\chi(D)u=\int_{\mathbb{R}^d}\tilde{h}(y)u(x-y)\d y,\\
&\Delta_k u\overset{\text{def}}{=}\varphi(2^{-k}D)u
=2^{kd}\int_{\mathbb{R}^d}{h}(2^ky)u(x-y)\d y,\,\,\text{if}\,\,k\geq 1.
\end{align*}
We also define the low-frequency cutoff
operators $S_k:$
$
S_ku =\sum\limits_{m\leq k-1}\Delta_m u.
$ This dyadic decomposition has nice quasi-orthogonality properties:
\begin{align*}
\Delta_m\Delta_ku=0 \,\,\text{if}\,\, |m-k|\geq 2,\,\,
\Delta_m(S_{k-1}u\Delta_ku)=0,\,\, \text{if}\,\,|m-k|\geq 5.
\end{align*}
Since $\chi$ is supported in a ball and $\varphi$ is supported in an annulus,  this leads to the following Bernstein inequalities.
\begin{lemma}[\cite{BCD2011}]\label{bern}
If $1\leq p\leq q\leq \infty$, then
\begin{align*}
  &\|\na^m \Delta_kf\|_{L^q}\lesssim 2^{k(m+d(1/p-1/q))}\|\Delta_k f\|_{L^p},\,\,k\geq -1,\\
 &2^{km}\|\Delta_k f\|_{L^p}\lesssim \|\na^m\Delta_kf\|_{L^p}\lesssim 2^{km}\|\Delta_k f\|_{L^p}, \,\,k\geq 0.
\end{align*}
\end{lemma}

For $s\in\mathbb{R}, 1\leq p, r \leq +\infty$, and $u\in \mathcal{S}'
(\mathbb{R}^d)$, we  define the Besov space ${B}_{p,r}^{s}$ with the norm
\begin{align*}
\|f\|_{{B}^s_{p,r}}\overset{\text{def}}{=} \Big ( \sum_{k\geq -1} 2^{ksr} \|{\Delta}_k f\|^r_{L^p}   \Big)^{\frac{1}{r}} \,\,\text{if}\,\,r<+\infty, \,\,\text{and}\,\,\|f\|_{{B}^s_{p,\infty}}\overset{\text{def}}{=} \sup_{k\geq -1} 2^{ks} \|{\Delta}_k f\|_{L^p}.
\end{align*}
It is well known that the Sobolev space $H^s$ coincides with the Besov space $B^s_{2,2}$.   In this paper, we let $\|f\|_{H^s}=\|f\|_{B^{s}_{2,2}}.$  We split the $H^s$-norm into two parts as follows:
 \begin{align*}
   \|f\|_{H^s}^2=\|f\|_{H^{s}_{\leq k_0}}^2+\|f\|_{H^{s}_{>k_0}}^2
 \end{align*}
 with
 \begin{align*}
\|f\|_{H^{s}_{\leq k_0}}\overset{\text{def}}{=} \Big ( \sum_{-1\leq k\leq k_0} 2^{2k{s}} \|{\Delta}_k f\|^2_{L^2}   \Big)^{\frac{1}{2}},\,\,
\|f\|_{H^{s}_{>k_0}}\overset{\text{def}}{=} \Big ( \sum_{k> k_0} 2^{2k{s}} \|{\Delta}_k f\|^2_{L^2}   \Big)^{\frac{1}{2}}.
\end{align*}
Direct calculation gives 
\begin{align}\label{ffk0-new}
    \|f\|_{H^{s}_{\leq k_0}}\lesssim \|f\|_{H^{s}}
    \end{align}
and 
\begin{align}
\label{l+}
\|f\|_{H^{s+l'}_{\leq k_0}}\lesssim
   2^{k_0l'} \| f \|_{H^{s}_{\leq k_0}},~\text{ for all}~ l'>0.
\end{align}
These two inequalities will be used extensively in the subsequent analysis.

   We then introduce the definition of frequency weight.
\begin{definition}\label{deft}
Let $\delta>0 $. An acceptable frequency weight $\{\omega_k\}$ is defined as a sequence satisfying $ 1\leq \omega_{k}\leq \omega_{k+1}\leq 2^{\delta}\omega_{k}$ for $k\geq -1$. We denote it by $\omega\in AF(\delta)$.
\end{definition}
With an acceptable frequency weight $\{\omega_k\}$, we slightly modulate the inhomogeneous Besov spaces in the following way:
for $s\in \mathbb{R}$, we define ${B}_{p,r}^{s}(\omega)$ with the norm
\begin{align*}
\|f\|_{{B}^s_{p,r}(\omega)}\overset{\text{def}}{=} \Big ( \sum_{k\geq -1} \omega_k^r2^{ksr} \|{\Delta}_k f\|^r_{L^p}   \Big)^{\frac{1}{r}} \,\,\text{if}\,\,r<+\infty,  \,\,\text{and}\,\,\|f\|_{{B}^s_{p,\infty}(\omega)}\overset{\text{def}}{=} \sup_{k\geq -1} \omega_k2^{ks} \|{\Delta}_k f\|_{L^p}.
\end{align*}
We define
$$\|f\|_{H^s(\omega)}=\|f\|_{{B}_{2,2}^{s}(\omega)}=\Big ( \sum\limits_{k\geq -1} \omega_k^22^{2ks} \|{\Delta}_k f\|^2_{L^2}   \Big)^{\frac{1}{2}}.$$
The properties of $B^s_{p,r}(\omega)$ have a lot in common with those of $B^{s}_{p,r}$.
 \begin{lemma}[\cite{BCD2011}]\label{pro-bes}
We have
\begin{enumerate}
\item If $s\in \mathbb{R}$, $1\leq p,r\leq \infty$, then  $\|\nabla f\|_{B^{s-1}_{p,r}}\lesssim \| f\|_{B^{s}_{p,r}}$,
\item
 If $1\leq p_1\leq p_2\leq \infty$ and $1\leq r_1\leq r_2\leq \infty$, then $B^{s}_{p_1,r_1}\hookrightarrow B^{s-d(1/p_1-1/p_2)}_{p_2,r_2},$
 \item If $s_2< s_1$ and $1\leq p,r_1,r_2\leq \infty$, then $B^{s_1}_{p,r_1}\hookrightarrow B^{s_2}_{p,r_2},$
 \item If $1\leq p\leq \infty$, then $B^{d/p}_{p,1}\hookrightarrow L^\infty.$
 \end{enumerate}
\end{lemma}
It is obvious that the following inequalities are valid for $B^s_{p,r}(\omega)$.
 \begin{lemma}
Suppose that $\omega\in AF(\delta)$. We have
\begin{enumerate}
\item If $s\in \mathbb{R}$, $1\leq p,r\leq \infty$, then $\|\nabla f\|_{B^{s-1}_{p,r}(\omega)}\leq C \| f\|_{B^{s}_{p,r}(\omega)}$,
\item
 If $1\leq p_1\leq p_2\leq \infty$ and $1\leq r_1\leq r_2\leq
 \infty$, then $$\|f\|_{B^{s-d(1/p_1-1/p_2)}_{p_2,r_2}(\omega)}\leq C\|f\|_{B^{s}_{p_1,r_1}(\omega)},$$
 \item If $s_2<s_1$ and $1\leq p,r_1,r_2\leq\infty$, then $\|f\|_{B^{s_2}_{p,r_2}(\omega)}\leq C\|f\|_{B^{s_1}_{p,r_1}(\omega)},$
 \item If $s\in \mathbb{R}$,  $1\leq p,r\leq \infty$, then $\|f\|_{B^{s}_{p,r}}\leq C\|f\|_{B^{s}_{p,r}(\omega)}.$
 \end{enumerate}
In the above, the constant  $C$ is independent of $\omega.$
 \end{lemma}

Considering the product acting on  $f$ and $g$, we employ the Bony decomposition:
\begin{align}\label{Bony}
fg=T_fg + T_gf+ R(f,g),
\end{align}
where the paraproduct $f$ and the remainder $g$ are defined by
 \begin{align*}
 T_fg=\sum S_{l-1}f\Delta_lg,\,\,R(f,g)=\sum\Delta_l f \tilde{\Delta}_l g\triangleq \sum\Delta_l f (\sum_{|l'-l|\leq 1}\Delta_{l'}) g.
\end{align*}
\begin{lemma}[\cite{BCD2011}] \label{paraproduct}
For any $s\in\mathbb{R}$ and $1\leq p,r\leq \infty$,
the following estimates hold:
\begin{align*}
 &\|T_{f}g\|_{B_{p,r
}^{s}} \lesssim \|f\|_{L^\infty}\|\nabla g\|_{B_{p,r}^{s-1}},\\
&\|T_{f}g\|_{B_{p,r
}^{s}}
\lesssim \|f\|_{B_{\infty,\infty}^{t}} \|\nabla g\|_{B_{p,r}^{s-t-1}}, \, t<0.
\end{align*}
For any $s_1,s_2\in\mathbb{R}$ and $1\leq p,p_1,p_2,r,r_1,r_2\leq \infty$
such that $s_1+s_2>0, 1/p_1+1/p_2=1/p$ and $1/r_1+1/r_2=1/r$, we have
\begin{align*}
 \|R(f,g)\|_{B_{p,r
}^{s_1+s_2}} \lesssim \|f\|_{B_{p_1,r_1}^{s_1}}\|g\|_{B_{p_2,r_2}^{s_2}}.
\end{align*}
\end{lemma}
\begin{lemma} \label{paraproduct-omega}
Suppose that $\omega\in AF(\delta)$.
For any $s\in\mathbb{R}$ and $1\leq p,r\leq \infty$, the following estimates hold:
\begin{align*}
 &\|T_{f}g\|_{B_{p,r
}^{s}(\omega)} \leq C(\delta)\|f\|_{L^\infty}\|\nabla g\|_{B_{p,r}^{s-1}(\omega)},
\\
&\|T_{f}g\|_{B_{p,r
}^{s}(\omega)}
\leq C(\delta)\|f\|_{B_{\infty,\infty}^{t}} \|\nabla g\|_{B_{p,r}^{s-t-1}(\omega)}, \, t<0.
\end{align*}
For any $s_1,s_2\in\mathbb{R}$ and $1\leq p,p_1,p_2,r,r_1,r_2\leq \infty$
such that $s_1+s_2>0, 1/p_1+1/p_2=1/p$ and $1/r_1+1/r_2=1/r$, we have
\begin{align*}
 &\|R(f,g)\|_{B_{p,r
}^{s_1+s_2}(\omega)} \leq C(\delta) \|f\|_{B_{p_1,r_1}^{s_1}}\|g\|_{B_{p_2,r_2}^{s_2}(\omega)},\\&\|R(f,g)\|_{B_{p,r
}^{s_1+s_2}(\omega)} \leq C(\delta) \|f\|_{B_{p_1,r_1}^{s_1}(\omega)}\|g\|_{B_{p_2,r_2}^{s_2}}.
\end{align*}
The constant $C(\delta)$ is 
independent of the exact values of $\omega$ and depends only  on $\delta$.
\end{lemma}
\begin{proof}
The proof of this lemma is very similar to that of Lemma \ref{paraproduct}.  For the paraproduct operator $T$, due to the definition of $S_k$,  we see
\begin{align*}
\Delta_k T_fg =\sum\limits_{l\geq 1,|l-k|\leq 4}\Delta_k(S_{l-1}f\Delta_lg),
\end{align*}
from which it follows
\begin{align*}
\|\Delta_k T_fg \|_{L^{p}}
&\lesssim \sum\limits_{l\geq 1, |l-k|\leq 4} \|S_{l-1}f\|_{L^\infty}\|\Delta_lg\|_{L^p}.
\end{align*}
Since $|l-k|\leq 4$, by virtue of the definition of $\{\omega_k\}$, we see
$2^{-4\delta }\leq \omega_k/\omega_l\leq 2^{4\delta}.$ 
This leads to
\begin{align*}
2^{ks}\omega_{k}\|\Delta_k T_fg \|_{L^{p}}
&\leq C(\delta) \sum\limits_{l\geq 1, |l-k|\leq 4} 2^{ls}\omega_{l}\|S_{l-1}f\|_{L^\infty}\|\Delta_lg\|_{L^p}.
\end{align*}
According to Lemma \ref{bern}, we have $\|\Delta_lg\|_{L^p}\lesssim 2^{-l}\|\Delta_l\nabla g\|_{L^p}$ with $l\geq 1$.
By using $\|S_{l-1}f\|_{L^\infty}\leq \|f\|_{L^\infty}$ (see, Young’s inequality for the convolution) or  $\|S_{l-1}f\|_{L^\infty}\lesssim 2^{-lt}\|f\|_{B^{t}_{\infty,\infty}}$  with $t<0$ (see, Proposition 2.79 in \cite{BCD2011}), we obtain that
\begin{align*}
 2^{ks}\omega_{k}\|\Delta_k T_fg \|_{L^{p}}
&\leq C(\delta) \|f\|_{L^\infty}\sum\limits_{ |l-k|\leq 4}2^{l{(s-1)}}\omega_l \|\Delta_l\na g\|_{L^p},
\end{align*}
or alternatively,
\begin{align*}
 2^{ks}\omega_{k}\|\Delta_k T_fg \|_{L^{p}}
&\leq C(\delta)\|f\|_{B^{t}_{\infty,\infty}}\sum\limits_{ |l-k|\leq 4}2^{l(s-t-1)}\omega_l \|\Delta_l\nabla g\|_{L^p}.
\end{align*}
Then taking the $l^r$-norm on both sides leads to the desired continuity properties of $T$.

For the remainder operator $R$, we have
\begin{align*}
\Delta_kR(f,g)=\sum_{l\geq k-3}\Delta_{k}(\Delta_l f \tilde{\Delta}_l g),
\end{align*}
from which it follows
\begin{align*}
\|\Delta_k R(f,g) \|_{L^{p}}
&\lesssim \sum\limits_{l\geq k-3} \|\Delta_l f\|_{L^{p_1}}\|\tilde{\Delta}_lg\|_{L^{p_2}}.
\end{align*}
Note that if $l\geq k-3$, then $ \omega_k\leq2^{3\delta} \omega_l$.  This entails that
\begin{align*}
2^{k(s_{1}+s_{2})}
\omega_{k}
 \|\Delta_k R(f,g) \|_{L^{p}}
 \leq C(\delta)\sum_{l\geq k-3}
2^{(k-l)(s_{1}+s_{2})}2^{ls_1}
\|\Delta_l f\|_{L^{p_{1}}}
\omega_l2^{ls_{2}}\|\tilde{\Delta}_l g\|_{L^{p_{2}}}.
\end{align*}
Then we can take the $l^r$-norm on both sides to obtain the desired continuity properties of $R$. This completes the proof of the lemma.
\end{proof}
The following lemma plays  a fundamental role in this paper.
\begin{lemma}\label{frequency}
Suppose $f^n\rightarrow f$ in $H^s$. Then there exists a sequence $\{\omega_k\}$ of positive numbers which satisfies
\begin{align*}
    \omega\in AF(\tfrac{1}{2}),\, \,\text{and}\,\,\lim\limits_{k\rightarrow +\infty}\omega_k=+\infty,
\end{align*}
such that for all $n$, $f^n\in H^s(\omega)$. Moreover, we have,
$\sup\limits_{n}\|f^n\|_{H^s(\omega)}<\infty$.
\end{lemma}
\begin{proof}
The proof proceeds in essentially the same manner as Lemma 4.1 in \cite{KT2003}. For the reader's convenience, we provide a brief sketch of the proof here.
The assumption $f^n\rightarrow f$ in $H^s$ implies
\begin{align*}
\{2^{2ks} \|{\Delta}_k f^n\|^2_{L^2}\}_{k\geq -1}\rightarrow \{2^{2ks} \|{\Delta}_k f\|^2_{L^2}\}_{k\geq -1}\,\,{\text{in}}\,\,l^1.
\end{align*}
Then for all $m\in \mathbb{N}$, there exist a strictly monotone $\{\mathcal{N}_m\}$ such that
\begin{align*}
\sup_{n}\Big(\sum\limits_{k=\mathcal{N}_m}^{+\infty}2^{2ks} \|{\Delta}_k f^n\|^2_{L^2}\Big)\leq 2^{-2m}.
\end{align*}
 If $-1\leq k<\mathcal{N}_1$, we set $\omega_k=1$;  if $\mathcal{N}_m\leq k<\mathcal{N}_{m+1}$ ($m\in \mathbb{N}$), we set $\omega_k=2^{m/2}$. Then we  obtain that
\begin{align*}
\sup_n\Big(\sum_{k\geq -1}w_k^22^{2ks} \|{\Delta}_k f^n\|^2_{L^2}\Big)
=&\sup_{n}\Big(\sum_{k=-1}^{\mathcal{N}_1-1}2^{2ks} \|{\Delta}_k f^n\|^2_{L^2}+\sum_{m=1}^{+\infty}
\sum_{k=\mathcal{N}_m}^{\mathcal{N}_{m+1}-1}\omega_k^22^{2ks} \|{\Delta}_k f^n\|^2_{L^2}\Big)\\
\leq &
\sup_n\Big(\|f^n\|_{H^s}^2+\sum_{m=1}^{+\infty}2^m2^{-2m}\Big)<\infty.
\end{align*}
This completes the proof of the lemma.
\end{proof}
Letting $f^n\equiv f$, we have the following straightforward corollary.
\begin{corollary}\label{corollary-fre}
    Suppose $f$ in $H^s$. Then there exists a sequence $\{\omega_k\}$ of positive numbers which satisfies
\begin{align*}
    \omega\in AF(\tfrac{1}{2}),\, \,\text{and}\,\,\lim\limits_{k\rightarrow +\infty}\omega_k=+\infty,
\end{align*}
such that $f\in H^s(\omega)$. 
\end{corollary}

\subsection{Useful estimates}
In this paper, in order to control the $L^\infty_x$-norm of the fluid, we use the following lemma.
\begin{lemma}\label{unablau}
  Suppose that $s'>1/2$. Then we have
  \begin{align*}
    \|u\|_{L^\infty} \lesssim &\|\nabla u\|_{H^{s'}}.
  \end{align*}
  \end{lemma}
\begin{proof}
  By virtue of the Littlewood-Paley decomposition, we see
  \begin{align*}
    u=\Delta_{-1}u+\sum_{k\geq 0}\Delta_{k}u.
  \end{align*}
  According to Young's inequality for the convolution and the Sobolev embedding $\dot{H}^1(\mathbb{R}^3)\hookrightarrow L^6(\mathbb{R}^3)$, we get that
  \begin{align*}
   &\|\Delta_{-1}u\|_{L^\infty} \lesssim \|u\|_{L^6}\lesssim \|\nabla u\|_{L^2}
  \lesssim \|\nabla u\|_{H^{s'}}.
  \end{align*}
  Thanks to Lemma \ref{bern}, we have, for $s'>1/2,$ then
\begin{align*}
   &\sum_{k\geq 0}\|\Delta_k u\|_{L^\infty}
   \lesssim \sum_{k\geq 0}2^{k/2}\|\Delta_k\nabla u\|_{L^2}
   \lesssim \sum_{k\geq 0}2^{k(1/2-s')}2^{ks'}\|\Delta_k\nabla u\|_{L^2}
   \lesssim \|\nabla u\|_{H^{s'}}.
   \end{align*}
We thus get the desired result.
\end{proof}
We establish the following product estimates which will be  frequently used in the subsequent analysis.
\begin{lemma} \label{jb}
If $s>\frac{3}{2}$ and $s'-1\leq s\leq s'+1$, then we  have
  \begin{align*}
  &\|u\cdot\nabla u\|_{H^{s'-1}}
  \lesssim \|u\|_{H^{s'}}\|\nabla u\|_{H^{s'}},\,\,
\|j\times B\|_{H^{s'-1}} \lesssim \|j\|_{H^{s}}\|B\|_{H^{s}},\\
  &\|j\cdot \nabla j\|_{H^{s'-1}} \lesssim \|j\|_{H^{s}}\|\nabla j\|_{H^{s}},\,\,
 \|u\times B\|_{H^{s}} \lesssim \|\nabla u\|_{H^{s'}}\|B\|_{H^{s}},\\
  &\|u\cdot \nabla j\|_{H^{s}} \lesssim \|\nabla u\|_{H^{s'}}\|\nabla j\|_{H^{s}},\,\,
\|j\cdot \nabla u\|_{H^{s-1}} \lesssim \|j\|_{H^{s}}\|\nabla u\|_{H^{s'}}.
\end{align*}
Moreover, we also have \begin{align*}
&\|u\cdot \na j\|_{H^{s-1}}\lesssim \| u\|_{H^{s'}}\|\na j\|_{H^s},\,\,
  \|u\times B\|_{H^{s-1}}
  \lesssim \|u\|_{H^{s'}}\|B\|_{H^{s}},\\
  & \|u\times \partial_t B\|_{H^{s-1}}
  \lesssim \|\nabla u\|_{H^{s'}}\|\partial_t B\|_{H^{s-1}},\,\,
 \|T_{\partial_t u}B\|_{H^{s-2}}
  \lesssim \|\partial_t u\|_{H^{s'-1}}\|B\|_{H^{s}},\\
   &\|T_B {\partial_t u} \|_{H^{s'-1}}+\|R(\partial_t u,B)\|_{H^{s'-1}}
  \lesssim \|B\|_{H^{s}}\|\partial_t u\|_{H^{s'-1}},\,\,\|j\cdot \na j\|_{H^{s-1}}\lesssim \|j\|_{H^{s}}^2.
\end{align*}

  \end{lemma}
\begin{proof}
We deduce from $s>\frac{3}{2}$ and $s-1\leq s'\leq s+1$ that $s'>\frac{1}{2}$.
The proof of all these inequalities relies on Bony's decomposition \eqref{Bony},  Lemma \ref{pro-bes}, Lemma \ref{paraproduct} and Lemma \ref{unablau}. In fact, for $u\cdot\nabla u$, we have
\begin{align*}
 &\|T_u\nabla u\|_{H^{s'-1}}\lesssim \|u\|_{B^{-1}_{\infty,\infty}}\|\nabla u\|_{H^{s'}}\lesssim
  \|u\|_{H^{\frac{1}{2}}}\|\nabla u\|_{H^{s'}}\lesssim
  \|u\|_{H^{s'}}\|\nabla u\|_{H^{s'}},\\
 &\|T_{\nabla u}u\|_{H^{s'-1}}\lesssim
  \|\nabla u\|_{B^{-1}_{\infty,\infty}}\| u\|_{H^{s'}}
  \lesssim \|\na u\|_{H^{s'}}\|u\|_{H^{s'}},\\
  &\|R(u,\nabla u)\|_{H^{s'-1}}
\lesssim \|R(u,\nabla u)\|_{B^{s'+\frac{1}{2}}_{1,1}}\lesssim \|u\|_{H^{\frac{1}{2}}}\|\nabla u\|_{H^{s'}}
 \lesssim \|u\|_{H^{s'}}\|\nabla u\|_{H^{s'}}.
  \end{align*}
  This yields $\|u\cdot\nabla u\|_{H^{s'-1}}
  \lesssim \|u\|_{H^{s'}}\|\nabla u\|_{H^{s'}}.$ For $u\times B$, we have
\begin{align*}
   &\|T_u B\|_{H^{s}}+\|R(u,B)\|_{H^s}
   \lesssim \|u\|_{L^\infty}\|B\|_{H^{s}}
  \lesssim \|\nabla u\|_{H^{s'}}\|B\|_{H^{s}},\\
  &\|T_Bu\|_{H^{s}}\lesssim \|B\|_{L^\infty}\|\nabla u\|_{H^{s-1}}
  \lesssim \|B\|_{H^{s}}\|\nabla u\|_{H^{s'}},
  \end{align*}
This implies $\|u\times B\|_{H^{s}} \lesssim \|\nabla u\|_{H^{s'}}\|B\|_{H^{s}}$.
  Similar computations yield the remaining desired inequalities, and the details are omitted.
    \end{proof}
    Using Lemma \ref{paraproduct-omega} instead of Lemma \ref{paraproduct}, we can derive the following product estimates for weighted Besov spaces.
\begin{lemma} \label{jbjb} Suppose that $\omega\in AF(\delta)$.
If  $s>\frac{3}{2}$ and $s'-1\leq s\leq s'+1$, then we have \begin{align*}
  &\|u\cdot\nabla u\|_{H^{s'-1}(\omega)}
  \leq C(\delta)( \| u\|_{H^{s'}}\|\nabla u\|_{H^{s'}(\omega)}+\| u\|_{H^{s'}(\omega)}\|\nabla u\|_{H^{s'}}),\\
&\|j\times B\|_{H^{s'-1}(\omega)} \leq C(\delta)(\|j\|_{H^{s}}\|B\|_{H^{s}(\omega)}+\|j\|_{H^{s}(\omega)}\|B\|_{H^{s}}),\\
&\|j\cdot \nabla j\|_{H^{s'-1}(\omega)} \leq C(\delta) (\|j\|_{H^{s}}\|\nabla j\|_{H^{s}(\omega)}+\|j\|_{H^{s}(\omega)}\|\nabla j\|_{H^{s}}),\\
&\|u\times B\|_{H^{s}(\omega)} \leq C(\delta)( \|\nabla u\|_{H^{s'}}\|B\|_{H^{s}(\omega)}+ \|\nabla u\|_{H^{s'}(\omega)}\|B\|_{H^{s}}),\\
&\|u\cdot \nabla j\|_{H^{s}(\omega)} \leq C(\delta)( \|\nabla u\|_{H^{s'}}\|\nabla j\|_{H^{s}(\omega)}+\|\nabla u\|_{H^{s'}(\omega)}\|\nabla j\|_{H^{s}}),\\
&\|j\cdot \nabla u\|_{H^{s-1}(\omega)} \leq C(\delta)(\|j\|_{H^{s}}\|\nabla u\|_{H^{s'}(\omega)}+\|j\|_{H^{s}(\omega)}\|\nabla u\|_{H^{s'}}).
  \end{align*}
  \end{lemma}

\section{Global existences results}

\subsection{A priori estimates}

As mentioned in the Introduction, through the transformation  $u = (u^++u^-)/2$  and $j = (u^+-u^-)/2\varepsilon$, the system \eqref{NSM} can be rewritten equivalently as system \eqref{euqNSM}. In what follows, we will directly derive estimates for \eqref{euqNSM}.

The energy functions are defined as follows:
		\begin{align*}
        & \mathcal{E}^\varepsilon(t)=\| u^\varepsilon\|_{H^{s'}}^2+\|\varepsilon j^\varepsilon\|_{H^s}^2+\|E^\varepsilon\|_{H^s}^2+\| B^{\varepsilon}\|_{H^s}^2, \\
			& \mathcal{D}^\varepsilon(t)=\mu\|\na u^\varepsilon\|_{H^{s'}}^2+\mu\|\varepsilon\nabla j^\varepsilon\|_{H^s}^2+\tfrac{1}{\sigma}\|j^\varepsilon\|_{H^s}^2,\\
			& \mathcal{E}^\varepsilon_{\omega}(t)=\| u^\varepsilon\|_{H^{s'}(\omega)}^2+\|\varepsilon j^\varepsilon\|_{H^s(\omega)}^2+\|E^\varepsilon\|_{H^s(\omega)}^2+\| B^\varepsilon\|_{H^s(\omega)}^2, \\
			& \mathcal{D}^\varepsilon_{\omega}(t)=\mu\|\na u^\varepsilon\|_{H^{s'}(\omega)}^2+\mu\|\varepsilon\nabla j^\varepsilon\|_{H^s(\omega)}^2+\tfrac{1}{\sigma}\|j^\varepsilon\|_{H^s(\omega)}^2.
		\end{align*}
We derive energy estimates for sufficiently smooth solutions.
\begin{lemma}\label{lemm:apriori}Suppose that $\omega\in AF(\delta)$. Assume that $(u^\varepsilon,j^\varepsilon,E^\varepsilon,B^\varepsilon)$ is a sufficiently smooth solution to system \eqref{euqNSM}.
	Then there exists a positive constant $C$, independent of $\varepsilon$,  such that
	  \begin{align*}
			&\tfrac{1}{2}\tfrac{\d}{\d t}\mathcal{E}^\varepsilon(t)+\mathcal{D}^\varepsilon(t)\leq C \mathcal{E}^\varepsilon(t)\fab\mathcal{D}^\varepsilon(t).\end{align*}
There also exists a positive constant $\tilde{C}$, independent of $\varepsilon$ and dependent on $\delta$, such that 
\begin{align*}
&\tfrac{1}{2}\tfrac{\d}{\d t}\mathcal{E}^\varepsilon_{\omega}(t)+\mathcal{D}^\varepsilon_{\omega}(t)\leq \tilde{C}(\delta)( \mathcal{E}^\varepsilon(t)\fab\mathcal{D}^\varepsilon_{\omega}(t)+\mathcal{D}^\varepsilon(t)\fab\mathcal{E}^\varepsilon_{\omega}(t)
\fab\mathcal{D}^\varepsilon_{\omega}(t)\fab).
\end{align*}
\end{lemma}

\begin{proof}
For simplicity, we drop the index $\varepsilon$ of the solution $(u^\varepsilon,\ j^\varepsilon,\ B^\varepsilon,\ E^\varepsilon)$.
Applying $\Delta_k$ to \eqref{euqNSM} yields
\begin{align}\label{localized-sys}
	  \left\{
		\begin{array}{ll}
			&\partial_t\Delta_ku-\mu\Delta \Delta_ku+\nabla \Delta_kp=\Delta_k(j\times B-u\cdot \nabla u-\varepsilon^2j\cdot \nabla j),
		\\[1ex]
			&\varepsilon^2\partial_t\Delta_kj-\varepsilon^2\mu\Delta \Delta_kj+\frac{1}{\sigma}\Delta_kj+\nabla \Delta_k\tilde{p}-c\Delta_kE=\Delta_k(u\times B-\varepsilon^2u\cdot \nabla j-\varepsilon^2j\cdot \nabla u),
		\\[1ex]
			&\frac{1}{c}\partial_t\Delta_kE-\nabla \times \Delta_kB+\Delta_kj=0,
		\\[1ex]
			&\frac{1}{c}\partial_t\Delta_kB+\nabla \times \Delta_kE=0.
		\end{array}
		\right.
	\end{align}
Taking the $L^2$-inner product of \eqref{localized-sys} with $(\Delta_k u, \Delta_k j,\Delta_k E,\Delta_k B)$, together with  $\mathrm{div} u=0$ and $\mathrm{div} j=0$, we have
\begin{align}
	  &\tfrac{1}{2}\tfrac{\d}{\d t}\| \Delta_ku\|_{L^2}^2 +\mu\|\nabla \Delta_ku  \|_{L^2}^2=I_{1,k},\label{delta u}\\
	  &\tfrac{1}{2}\tfrac{\d}{\d t}\| \varepsilon \Delta_kj\|_{L^2}^2+
 \mu\|\varepsilon\nabla \Delta_k j \|_{L^2}^2 + \tfrac{1}{\sigma}\|\Delta_k j\|_{L^2}^2 -c\langle \Delta_kE, \Delta_kj\ra_{L^2}=I_{2,k}+I_{3,k},\no\\
&\tfrac{1}{2}\tfrac{\d}{\d t}
\|\Delta_k E\|_{L^2}^2-c\la \na\times \Delta_kB, \Delta_kE\ra_{L^2}+c\langle \Delta_kj, \Delta_kE\ra_{L^2}=0,\no\\
&\tfrac{1}{2}\tfrac{\d}{\d t}
\|\Delta_k B\|_{L^2}^2+c\la \na\times \Delta_kE, \Delta_kB\ra_{L^2}=0,\label{equ-b-eney}
\end{align}
where
\begin{align*}
 &I_{1,k}=\langle \Delta_k(j\times B-u\cdot \nabla u-\varepsilon^2j\cdot \nabla j), \Delta_ku\rangle_{L^2},\\  
 &I_{2,k}=\langle \Delta_k(u\times B-\varepsilon^2u\cdot \nabla j), \Delta_kj\rangle_{L^2},\,\,I_{3,k}=\langle \Delta_k(-\varepsilon^2 j\cdot \nabla u), \Delta_kj\rangle_{L^2}.
\end{align*}
There are convenient cancellations:
 \begin{align}
	&  -\la \na\times \Delta_kB, \Delta_kE\ra_{L^2}+\la \na\times \Delta_kE, \Delta_kB\ra_{L^2}=0,\label{cancek} \\
&-\langle \Delta_kE, \Delta_kj\ra_{L^2}+\langle \Delta_kj, \Delta_kE\ra_{L^2}=0.\no
	\end{align}
Hence, by adding the last three energy equations together, we conclude that
\begin{align}\label{delta jeb}
	  \tfrac{1}{2}\tfrac{\d}{\d t}
(\| \varepsilon \Delta_k&j\|_{L^2}^2+\|\Delta_kE\|_{L^2}^2+\| \Delta_k B\|_{L^2}^2)
+
 \mu\|\varepsilon\nabla \Delta_k j \|_{L^2}^2 + \tfrac{1}{\sigma}\|\Delta_k j\|_{L^2}^2=I_{2,k}+I_{3,k}.
\end{align}
Multiplying both sides of \eqref{delta u} and \eqref{delta jeb} by $2^{2ks'}$ and $2^{2ks}$, respectively, and then adding these two inequality together and summing over $k\geq -1$, we obtain
	\begin{align}\label{us'}
	  &\tfrac{1}{2}\tfrac{\d}{\d t}\mathcal{E}(t)+\mathcal{D}(t)
	  =\sum_{k\geq -1}(2^{2ks'}I_{1,k}+2^{2ks}I_{2,k}+2^{2ks}I_{3,k}).
	\end{align}

 We split the sum of $I_{1,k}$ and $I_{3,k}$ into low and high frequency parts. More precisely, if $k=-1$, according to the proof of Lemma \ref{unablau}, $\|\Delta _{-1}f\|_{L^\infty}\lesssim\|\na f\|_{L^2},$ we  get
 \begin{align*}
 2^{-2s'}I_{1,-1}
 \lesssim &\|j\times B-u\cdot \nabla u-\varepsilon^2j\cdot \nabla j\|_{L^1}\|\Delta_{-1} u\|_{L^\infty}\\
 \lesssim &(\|j\|_{L^2}\|B\|_{L^2}+\|u\|_{L^2}\|\na u\|_{L^2}
 +\|\varepsilon j\|_{L^2}\|\varepsilon\na j\|_{L^2})\|\nabla u\|_{L^2}.
	\end{align*}
and
\begin{align*}
 2^{-2s}I_{3,-1}
 \lesssim \|-\varepsilon^2j\cdot \nabla u\|_{L^1}\|\Delta_{-1} j\|_{L^\infty}
 \lesssim \|\varepsilon j\|_{L^2}\|\nabla u\|_{L^2}\|\varepsilon \nabla j\|_{L^2}.
	\end{align*}
Whereas, if $k\geq 0$, by Lemma \ref{bern},
$
  \|\Delta_{k} f\|_{L^2}\lesssim 2^{-k}\|\Delta_k \nabla f\|_{L^2},
$
 we obtain
\begin{align*}
  \sum_{k\geq 0}2^{2ks'}I_{1,k}
 \lesssim &\sum_{k\geq 0}2^{k(s'-1)} \| \Delta_k(j\times B-u\cdot \nabla u-\varepsilon^2j\cdot \nabla j)\|_{L^2}2^{ks'}\|\Delta_k \nabla u\|_{L^2}\\
 \lesssim &(\|j\times B\|_{H^{s'-1}}+\|u\cdot \nabla u\|_{H^{s'-1}}+\|\varepsilon^2j\cdot \nabla j\|_{H^{s'-1}})\|\nabla u\|_{H^{s'}},
\end{align*}
and 
\begin{align*}
 \sum_{k\geq 0}2^{2ks}I_{3,k}
 \lesssim \sum_{k\geq 0}2^{k(s-1)} \| \Delta_k(-\varepsilon^2j\cdot \nabla u)\|_{L^2}2^{ks}\|\Delta_k \nabla j\|_{L^2}
 \lesssim
 \|\varepsilon j\cdot \nabla u\|_{H^{s-1}}\|\varepsilon \nabla j\|_{H^{s}}.
\end{align*}
For $I_{2,k}$, a direct calculation yields
\begin{align*}
\sum_{k\geq -1}2^{2ks}I_{2,k}
\lesssim &\|u\times B\|_{H^{s}}\| j\|_{H^{s}}+\|\varepsilon u\cdot \nabla j\|_{H^{s}}\| \varepsilon j\|_{H^{s}}.
	\end{align*}
Substituting the above inequalities into \eqref{us'} and using Lemma \ref{jb},  we ultimately get
 \begin{align}\label{energy}
			&\tfrac{1}{2}\tfrac{\d}{\d t}\mathcal{E}(t)+\mathcal{D}(t)\lesssim \mathcal{E}(t)\fab\mathcal{D}(t).
\end{align}

As for the weighted  case, it simply involves multiplying both sides of \eqref{delta u} and \eqref{delta jeb} by $\omega_k^22^{2ks'}$ and $\omega_k^22^{2ks}$, respectively. Then
following along the same lines as in the proof of \eqref{energy} and using Lemma \ref{jbjb}, we obtain
\begin{align*}
&\tfrac{1}{2}\tfrac{\d}{\d t}\mathcal{E}_{\omega}(t)+\mathcal{D}_{\omega}(t)\leq \tilde{C}(\delta)( \mathcal{E}(t)\fab\mathcal{D}_{\omega}(t)+\mathcal{D}(t)\fab\mathcal{E}_{\omega}(t)
\fab\mathcal{D}_{\omega}(t)\fab).
\end{align*}
This completes the whole proof of the lemma.
\end{proof}

\subsection{Global existence for small data}
\label{sub:global_existence_for_small_data}
\
\newline
\indent
{\bf Proof of Theorem \ref{thm:GNSM}.}
{\bf First Step: the initial data.}
Suppose that the initial data $u^{\varepsilon,in}\in H^{s'}, j^{\varepsilon,in},E^{\varepsilon,in},B^{\varepsilon,in}\in H^{s}.$
According to Corollary \ref{corollary-fre},  there exists a sequence  $\{\omega_k\}$  satisfying
\begin{align}\label{omega-property}
    \omega\in AF(\tfrac{1}{2}),\, \,\text{and}\,\,\lim\limits_{k\rightarrow +\infty}\omega_k=+\infty,
\end{align}
such that $u^{\varepsilon,in}\in H^{s'}(\omega)$, $j^{\varepsilon,in},E^{\varepsilon,in},B^{\varepsilon,in}\in H^{s}(\omega)$.

Denote \begin{align*}
&\mathcal{E}^{\varepsilon,in}= \|u^{\varepsilon,in}\|_{H^{s'}}^2 + \| \varepsilon j^{\varepsilon,in}\|_{H^s}^2
		  + \|E^{\varepsilon,in}\|_{H^{s}}^2
		  + \|B^{\varepsilon,in}\|_{H^{s}}^2,\\
 &\mathcal{E}^{\varepsilon,in}_{\omega}= \|u^{\varepsilon,in}\|_{H^{s'}(\omega)}^2 + \| \varepsilon j^{\varepsilon,in}\|_{H^s(\omega)}^2
		  + \|E^{\varepsilon,in}\|_{H^{s}(\omega)}^2
		  + \|B^{\varepsilon,in}\|_{H^{s}(\omega)}^2.
\end{align*}

{\bf {Second step: the Friedrichs approximation.}} We use the Friedrichs method (see, for instance, \cite{BCD2011}) to construct  the approximate solutions to \eqref{euqNSM}.
Define the cutoff mollifier $\mathcal{J}_m:$
	$$
	  \mathcal{J}_m f=\mathcal{F}^{-1}
(I_{B(0,m)}(\xi)\hat{f}(\xi)).	$$
We consider the approximate system of \eqref{euqNSM}  as follows:
	\begin{align*}
	  \left\{
		\begin{array}{ll}
			\partial_tu^m+\mathcal{J}_m(\mathcal{J}_mu^m\cdot \nabla \mathcal{J}_mu^m)+\varepsilon^2\mathcal{J}_m(\mathcal{J}_mj^m\cdot \nabla \mathcal{J}_mj^m)-\mu\Delta \mathcal{J}_mu^m\\[1ex] \hspace{6cm}=-\nabla p^m+\mathcal{J}_m(\mathcal{J}_mj^m\times \mathcal{J}_mB^m),
		\\[1ex]
			\varepsilon^2\partial_t j^m+\varepsilon^2\mathcal{J}_m(\mathcal{J}_mu^m\cdot \nabla \mathcal{J}_mj^m+\mathcal{J}_mj^m\cdot \nabla \mathcal{J}_mu^m)-\varepsilon^2\mu\Delta \mathcal{J}_mj^m+\frac{1}{\sigma}\mathcal{J}_mj^m \\[1ex] \hspace{6cm}=-\nabla \tilde{p}^m+c\mathcal{J}_mE^m+\mathcal{J}_m(\mathcal{J}_mu^m\times \mathcal{J}_mB^m),
		\\[1ex]
			\frac{1}{c}\partial_tE^m-\nabla \times \mathcal{J}_m B^m=-\mathcal{J}_mj^m,
		\\[1ex]
			\frac{1}{c}\partial_tB^m+\nabla \times \mathcal{J}_mE^m=0,
		\\[1ex]
			\div u^m=0,\ \div j^m=0,\ \div E^m=0,\ \div B^m=0,
		\end{array}
		\right.
	\end{align*}
with the initial data
\begin{align*}
	  (u^{m,in},j^{m,in},E^{m,in},B^{m,in})=
(\mathcal{J}_mu^{\varepsilon,in},\mathcal{J}_mj^{\varepsilon,in},
\mathcal{J}_mE^{\varepsilon,in},\mathcal{J}_mB^{\varepsilon,in}).
	\end{align*}

We solve the above ordinary differential equations in $(L^2)^4$.  There exists a positive  maximal time $T^m$,  such that this approximate system  admits a unique solution $(u^m, j^m,E^m,B^m)\in \big(C([0,T^m);L^2)\big)^4$.
Since $\mathcal{J}_m^2=\mathcal{J}_m,$  $\mathcal{J}_m(u^m,j^m,E^m,B^m)$ is also a solution. Then the uniqueness then implies $(u^m,j^m,E^m,B^m)=(\mathcal{J}_mu^m,\mathcal{J}_mj^m,\mathcal{J}_mE^m,
\mathcal{J}_mB^m)$. Removing $\mathcal{J}_m$ in front of $(u^m,j^m,E^m,B^m)$, the system reads
	\begin{align}\label{asym-sys}
	  \left\{
		\begin{array}{ll}
			\partial_tu^m+\mathcal{J}_m(u^m\cdot \nabla u^m)+\varepsilon^2\mathcal{J}_m(j^m\cdot \nabla j^m)-\mu\Delta u^m=-\nabla p^m+\mathcal{J}_m(j^m\times B^m),
		\\[1ex]
			\varepsilon^2\partial_t j^m+\varepsilon^2\mathcal{J}_m(u^m\cdot \nabla j^m+j^m\cdot \nabla u^m)-\varepsilon^2\mu\Delta j^m+\frac{1}{\sigma}j^m \\[1ex] \hspace{6cm}= -\nabla \tilde{p}^m+cE^m+\mathcal{J}_m(u^m\times B^m),
		\\[1ex]
			\frac{1}{c}\partial_tE^m-\nabla \times B^m=-j^m,
		\\[1ex]
			\frac{1}{c}\partial_tB^m+\nabla \times E^m=0.
\\[1ex]
			\div u^m=0,\ \div j^m=0,\ \div E^m=0,\ \div B^m=0.
		\end{array}
		\right.
	\end{align}
It is straightforward to verify that the energy estimate holds
	\begin{align*}
	  \tfrac{1}{2}\tfrac{\d}{\d t}(\|u^m\|_{L^2}^2
	 & +\|\varepsilon j^m\|_{L^2}^2
	  +\|E^m\|_{L^2}^2+\|B^m\|_{L^2}^2)\\&+
	  \mu\|\nabla u^m\|_{L^2}^2+\mu\|\varepsilon\nabla j^m\|_{L^2}^2+\tfrac{1}{\sigma}\|j^m\|_{L^2}^2=0.
	\end{align*}
This implies  the $L^2$ norm of $(u^m, j^m,E^m,B^m)$ is controlled, and consequently, we obtain $T^m=\infty$.

{\bf {Third step: uniform estimates.}}
By virtue of  the low-frequency cut-off property of $\mathcal{J}_m$, for any $s_1\in\mathbb{R}$, $(u^m,j^m,E^m,B^m)\in (C([0,+\infty);H^{s_1}))^4$.
Furthermore, again using $\mathcal{J}_m$ is an orthogonal projector $L^2$, the inequalities of Lemma \ref{lemm:apriori} still hold. Specifically,  letting $\mathcal{E}^m,\mathcal{D}^m,\mathcal{E}^m_{\omega},\mathcal{D}^m_{\omega}$ be the corresponding energy functions for $(u^m,j^m,E^m,B^m)$, We  have
 \begin{align}
			&\tfrac{1}{2}\tfrac{\d}{\d t}\mathcal{E}^m(t)+\mathcal{D}^m(t)\leq C \mathcal{E}^m(t)\fab\mathcal{D}^m(t),\label{m-ene}\end{align}
            and 
            \begin{align}\label{kuaidian}
&\tfrac{1}{2}\tfrac{\d}{\d t}\mathcal{E}^m_{\omega}(t)+\mathcal{D}^m_{\omega}(t)\leq \tilde{C}(\tfrac{1}{2})( \mathcal{E}^m(t)\fab\mathcal{D}^m_{\omega}(t)+\mathcal{D}^m(t)\fab\mathcal{E}^m_{\omega}(t)
\fab\mathcal{D}^m_{\omega}(t)\fab).
\end{align}
Here $\delta=\tfrac{1}{2}.$ We thus denote $\tilde{C}(\tfrac{1}{2})$ simply by $C$. According to Young's inequality, we have
\begin{align*}
C\mathcal{D}^m(t)\fab\mathcal{E}^m_{\omega}(t)
\fab\mathcal{D}^m_{\omega}(t)\fab\leq C\mathcal{D}^m(t)\mathcal{E}^m_{\omega}(t)
+\tfrac{1}{2}\mathcal{D}^m_{\omega}(t).
\end{align*}
Substituting  this inequality into \eqref{kuaidian}, we obtain
\begin{align}
&\tfrac{\d}{\d t}\mathcal{E}^m_{\omega}(t)+\mathcal{D}^m_{\omega}(t)\leq C( \mathcal{E}^m(t)\fab\mathcal{D}^m_{\omega}(t)+\mathcal{D}^m(t)
\mathcal{E}^m_{\omega}(t)),\label{m-omega-ene}
		\end{align}
 Based on the initial assumption, we observe that:
\begin{align*}
    \mathcal{E}^{m,in}=\|\mathcal{J}_mu^{\varepsilon,in}\|_{H^{s'}}^2 + \| \varepsilon \mathcal{J}_mj^{\varepsilon,in}\|_{H^s}^2
		  + \|\mathcal{J}_mE^{\varepsilon,in}\|_{H^{s}}^2
		  + \|\mathcal{J}_mB^{\varepsilon,in}\|_{H^{s}}^2
\leq \mathcal{E}^{\varepsilon,in}\leq \kappa_0.
\end{align*}
If $\kappa_0$ is  sufficiently small such that $ C\kappa_0 \fab\leq \frac{1}{4}$.
then  $C (\mathcal{E}^{m,in})\fab\leq \frac{1}{4}$. By the continuity of $ \mathcal{E}^m(t)$, we define
\begin{align*}
  T^*=\sup\{t;\,\, \forall\, 0\leq t'\leq t, C \mathcal{E}^m(t')\fab\leq \tfrac{1}{2}\}.
\end{align*}
If $T^*$ is finite,
 then \eqref{m-ene} yields for all $t\in[0,T^*]$
 $
			\tfrac{\d}{\d t}\mathcal{E}^m(t)\leq 0.
$
 This  leads to $C \mathcal{E}^m(T^*)\fab\leq C (\mathcal{E}^{m,in})\fab\leq \frac{1}{4}$. Again, by virtue of the continuity of $ \mathcal{E}^m(t)$, there exists a positive time $t^*$, such that for all
$ t'\in[T^*, T^*+t^*]$,  $C \mathcal{E}^m(t')\fab\leq \tfrac{1}{2}.$
This contradicts the definition of $T^*$. Hence, $T^*=\infty$. We now conclude that,
\begin{align}\label{small-solu}
  C \mathcal{E}^m(t)\fab\leq \tfrac{1}{2},\,\, {\text {for all}}\,\, t\geq 0.
  \end{align}
Plugging \eqref{small-solu} into 	\eqref{m-ene} yields
\begin{align*}
			\tfrac{\d}{\d t}\mathcal{E}^m(t)+\mathcal{D}^m(t)\leq 0.
\end{align*}
Integrating with respect to time leads us to
\begin{align}\label{uni-1}
  \sup_{t\geq 0}\mathcal{E}^m(t)\leq \mathcal{E}^{m,in}\leq \kappa_0,\,\,
  \int_0^{+\infty} \mathcal{D}^m(t)\d t \leq \mathcal{E}^{m,in}\leq \kappa_0.
\end{align}
Next, plugging \eqref{small-solu} into \eqref{m-omega-ene},  we have
\begin{align*}
&\tfrac{\d}{\d t}\mathcal{E}^m_{\omega}(t)+\tfrac{1}{2}\mathcal{D}^m_{\omega}(t)\leq C\mathcal{D}^m(t)
\mathcal{E}^m_{\omega}(t).
		\end{align*}
Using Gronwall's inequality and \eqref{uni-1}, we have
\begin{align}\label{uni-omega}
 \sup_{t\geq 0} \mathcal{E}^m_{\omega}(t)+\int_0^{+\infty} \mathcal{D}^m_{\omega}(t) \d t\leq 3e^{C\int_0^{+\infty} \mathcal{D}^m(t)\d t}\mathcal{E}^{m,in}_{\omega}\leq 3e^{C\kappa_0}\mathcal{E}^{\varepsilon,in}_{\omega},
\end{align}
where we have used
\begin{align*}
    \mathcal{E}^{m,in}_{\omega}\overset{\text{def}}{=} & \|\mathcal{J}_mu^{\varepsilon,in}\|_{H^{s'}
    (\omega)}^2 + \|\varepsilon \mathcal{J}_mj^{\varepsilon,in}\|_{H^s(\omega)}^2
		 \\& + \|\mathcal{J}_mE^{\varepsilon,in}\|_{H^{s}(\omega)}^2
		  + \|\mathcal{J}_mB^{\varepsilon,in}\|_{H^{s}(\omega)}^2
\leq  \mathcal{E}^{\varepsilon,in}_{\omega}.
\end{align*}

{\bf Fourth step: compactness and convergence.}
We return to system \eqref{asym-sys}.
 According to Lemma \ref{jb},  along with  the uniform bound \eqref{uni-1},
  it can be verified that
 $$\partial_tu^m\,\text{ is bounded in}\, L^2(\mathbb{R}^+;H^{s'-1}),\,
 \partial_t j^m\,\text{ is bounded in}\, L^\infty(\mathbb{R}^+;H^{s})+ L^2(\mathbb{R}^+;H^{s-1}),$$
 $$\partial_t E^m\,\text{ is bounded in}\, L^\infty(\mathbb{R}^+;H^{s-1})+ L^2(\mathbb{R}^+;H^{s}),\, \partial_t B^m\,\text{ is bounded in}\, L^\infty(\mathbb{R}^+;H^{s-1}).$$
 By virtual of the Aubin-Lions-Simon Theorem and the Cantor diagonal process, up to
a subsequence,  the sequence $(u^m, j^m, E^m, B^m)$ converges in $\mathcal{S}'$ to  some divergence-free function $(u^\varepsilon, j^\varepsilon, E^\varepsilon, B^\varepsilon)$ such that for any $T>0$, as $m\rightarrow \infty,$
\begin{align*}
  &u^m\rightarrow u^\varepsilon \,\,{\text{strongly in }} C([0,T];H^{s'-\iota}_{loc}),\,\, j^m\rightarrow j \,\,{\text{strongly in }} C([0,T];H^{s-\iota}_{loc}),\\
 & E^m\rightarrow E^\varepsilon \,\,{\text{strongly in }} C([0,T];H^{s-\iota}_{loc}),\,\,
  B^m\rightarrow B^\varepsilon \,\,{\text{strongly in }} C([0,T];H^{s-\iota}_{loc}),
\end{align*}
with $\iota>0$.
Then it is straightforward to pass to the limit in \eqref{asym-sys} and conclude that $(u^\varepsilon, j^\varepsilon, E^\varepsilon, B^\varepsilon)$ is indeed a solution of \eqref{euqNSM}.
According to the Fatou property (see, for example, Theorem 2.72 in \cite{BCD2011}),  for any $k\geq -1$,  we have $$\|(\Delta_ku^\varepsilon,\Delta_kj^\varepsilon,\Delta_kE^\varepsilon,\Delta_kB^\varepsilon)
\|_{L^2}\leq \liminf\limits_{m\rightarrow \infty}\|(\Delta_k u^m,\Delta_kj^m,\Delta_kE^m,\Delta_kB^m)\|_{L^2}.$$
This yields \begin{align*}
 &u^\varepsilon \in L^\infty(\mathbb{R}^+;{H}^{s'}),\,\,
 \nabla u^\varepsilon \in L^2(\mathbb{R}^+;{H}^{s'}),\,\,
	   j^\varepsilon \in L^\infty(\mathbb{R}^+;{H}^{s}),\\
 &\nabla j^\varepsilon\in L^2(\mathbb{R}^+;{H}^{s}),\,\,
	   E^\varepsilon \in L^\infty(\mathbb{R}^+;H^s),\,\,\ B^\varepsilon\in L^\infty(\mathbb{R}^+;H^s).
\end{align*}
Moreover, we deduce from the uniform bounds \eqref{uni-1} and \eqref{uni-omega} that
\begin{align}\label{uni-2}
	&\sup\limits_{t\geq 0}\mathcal{E}^\varepsilon(t)+
  \int_0^{+\infty} \mathcal{D}^\varepsilon(t)\d t \leq 2\kappa_0,\\
\label{uni-omega-ep}
& \sup\limits_{t\geq 0}\mathcal{E}^\varepsilon_{\omega}(t)+\int_0^{+\infty} \mathcal{D}^\varepsilon_{\omega}(t) \d t \leq 3e^{C\kappa_0}\mathcal{E}^{\varepsilon,in}_\omega.
\end{align}

Next, we show that $u^{\varepsilon}$ 
is continuous in time with values in $H^{s'}$. 
In fact, on the one hand,
since $\partial_tu^\varepsilon \in L^2(\mathbb{R}^+;H^{s'-1})$,  this implies $u^\varepsilon \in C(\mathbb{R}^+;H^{s'-1})$.  Therefore,  for any $k\geq -1,$ $S_ku^\varepsilon \in C(\mathbb{R}^+;H^{s'})$.
 One the other hand, the increasing property of $\{\omega_k\}$ entails that
 \begin{align}\label{w}
   \|u^\varepsilon\|_{H^{s'}_{>k}}^2=\sum_{m> k } 2^{ms'}\|\Delta_m u^\varepsilon\|_{L^2}^2
   \leq \frac{1}{\omega_k^2} \sum_{m> k }\omega_m^2 2^{ms'}\|\Delta_m u^\varepsilon\|_{L^2}^2\leq \tfrac{1} {\omega_k^2}\|u^\varepsilon\|_{H^{s'}  (\omega)}^2.\end{align}
Hence, by virtue of $\lim\limits_{k\rightarrow +\infty}\omega_k=+\infty$, together with   \eqref{uni-omega-ep}, we have
\begin{align*}
  \|u^\varepsilon-S_{k+1} u^\varepsilon\|_{L^\infty(\mathbb{R}^+;H^{s'})}\lesssim  \|u^\varepsilon\|_{L^\infty(\mathbb{R}^+;H^{s'}_{>k})}\rightarrow 0,\,{\text{as}}\, k\rightarrow +\infty.
\end{align*}
This implies $u^\varepsilon\in C(\mathbb{R}^+;H^{s'})$. Similar arguments also show that $j^\varepsilon, E^\varepsilon$, and $B^\varepsilon$ are continuous in time with values in $H^s$.

{\bf Last step: uniqueness.} The proof of uniqueness is standard. Suppose that there are two solutions with the same initial data. We consider the difference between these two solutions in a small norm. We omit the details here. The reader may refer to, for example,  \cite{GISY2014,Z2021}.
 \section{Further bounds for the low-frequency part}
   This section is devoted to showing that
$\varepsilon^2\partial_t j^\varepsilon\in L^2({\mathbb{R}^+;H^{s}_{\leq k_0}})$. It is worth mentioning that without the truncation, we do not expect $\varepsilon^2\partial_t j^\varepsilon$ can belong to $L^2(\mathbb{R}^+;H^{s})$.
\begin{lemma}\label{hf}
Assume that $(u^\varepsilon,j^\varepsilon,E^\varepsilon,B^\varepsilon)$ is the solution obtained in Theorem \ref{thm:GNSM}. Let $0< \varepsilon \leq 1$.  For any $k_0\geq -1$, there exists a positive constant $C$, independent of $\varepsilon$ and $k_0$, such that
	\begin{align*}
   \int_0^{+\infty}\|\varepsilon^2\partial_t j^\varepsilon\|_{H^{s}_{\leq k_0}}^2\d t\leq C( \|\varepsilon j^{\varepsilon,in}\|_{H^{s}}^2+\varepsilon2^{2k_0}).
 \end{align*}

\end{lemma}
\begin{proof}
For simplicity, we drop the index $\varepsilon$ of the solution $(u^\varepsilon,\ j^\varepsilon,\ B^\varepsilon,\ E^\varepsilon)$. \\
{\bf {Step 1. the
decay and dissipative properties of the electromagnetic field.}}
Substituting \eqref{euqNSM}-(2) in \eqref{euqNSM}-(3), we get
\begin{align}\label{e-damp}
\tfrac{1}{c}\partial_tE-\nabla \times B+\sigma cE=\sigma\nabla \tilde{p}+\sigma \tilde{J},
\end{align}
with $\tilde{J}=\varepsilon^2\partial_tj-\varepsilon^2\mu\Delta j-u\times B+\varepsilon^2u\cdot \nabla j+\varepsilon^2j\cdot \nabla u.$ There is a decay term $\sigma cE$.
Next, the fact that $B$ is divergence free yields
	$$
	  \na\times (\na\times B)=-\Delta B.
	$$
Hence,  we derive from \eqref{euqNSM}-(4) and \eqref{e-damp} that
	\begin{align}\label{ttb}
			 \tfrac{1}{c^2}\partial_{tt} B-\Delta B+\sigma  \partial_t B=-\sigma \na\times \tilde{J},
	\end{align}
which is a nonlinear wave equation with a damping term $\sigma\partial_t B$.

We apply $\Delta_k$ to \eqref{e-damp}, and take the $L^2$-inner product with $\Delta_k E$.  Using \eqref{equ-b-eney} and \eqref{cancek}, we have
\begin{align*}
-\la \na\times \Delta_kB, \Delta_kE\ra_{L^2}
=-\la \na\times \Delta_kE, \Delta_kB\ra_{L^2}= \tfrac{1}{2c}\tfrac{\d}{\d t}\|\Delta_k B\|_{L^2}^2.
\end{align*}
 Together with the fact 
that $\mathrm{\div} E=0$,  
we conclude that \begin{align*}
	  \tfrac{1}{2c}\tfrac{\d}{\d t}(\|\Delta_k E\|_{L^2}^2+\|\Delta_k B\|_{L^2}^2)
	  +\sigma c\|\Delta_k E\|_{L^2}^2 =\sigma\la \Delta_k \tilde{J},\Delta_k E\ra_{L^2}.
	\end{align*}
Using H\"{o}lder's inequality and Young's inequality, we have 
\begin{align*}
   \sigma\la \Delta_k \tilde{J},\Delta_k E\ra_{L^2}\leq \frac{\sigma c}{2}\|\Delta_k E\|_{L^2}^2 +\frac{\sigma}{2c}\| \Delta_k \tilde{J}\|_{L^2}^2.
\end{align*}
 This leads to \begin{align}\label{further-1}
	  &\tfrac{1}{c}\tfrac{\d}{\d t}(\|\Delta_k E\|_{L^2}^2+\|\Delta_k B\|_{L^2}^2)
	  +\sigma c\|\Delta_k E\|_{L^2}^2 \leq\frac{\sigma}{c}\| \Delta_k \tilde{J}\|_{L^2}^2.
\end{align}
Similarly, applying $\Delta_k$ to \eqref{ttb}, taking the $L^2$-inner product with $\Delta_k\partial_t B$, and  using
	\begin{align*}
	&-\sigma \la\na\times \Delta_k \tilde{J},\Delta_k\partial_t B\ra_{L^2}\leq \tfrac{\sigma}{2} \|\na \Delta_k \tilde{J}\|_{L^2}^2+\tfrac{\sigma}{2}\|\Delta_k\partial_t B\|_{L^2}^2,
	\end{align*}
we get
\begin{align}
	 &\tfrac{\d}{\d t}(\tfrac{1}{c^2}\|\Delta_k\partial_t B\|_{L^2}^2+\|\Delta_k\na B\|_{L^2}^2)
	+\sigma \|\Delta_k\partial_t B\|_{L^2}^2\leq \sigma\|\na \Delta_k \tilde{J}\|_{L^2}^2.\label{further-2}
	\end{align}
Next, we take the $L^2$-inner product with $\Delta_k B$ in the equation \eqref{ttb}. Performing integration by parts yields
\begin{align*}
 \la \Delta_k \partial_{tt} B, \Delta_k B\ra=\tfrac{\d}{\d t}\la \Delta_k \partial_t B, \Delta_k B\ra -\|\Delta_k\partial_t B\|_{L^2}^2.
\end{align*}
Again using \eqref{cancek}, together with H\"{o}lder's inequality and Young's inequality, we have
	\begin{align*}
	& -\sigma \la\na\times \Delta_k \tilde{J},\Delta_k B\ra_{L^2}=  -\sigma \la\Delta_k \tilde{J},\na\times \Delta_k B\ra_{L^2}
\leq  \tfrac{\sigma^2}{2} \|\Delta_k \tilde{J}\|_{L^2}^2+\tfrac{1}{2}\|\na \Delta_k B\|_{L^2}^2.
	\end{align*}
we finally obtain that
\begin{align}
		\tfrac{\d}{\d t}(\tfrac{2}{c^2}\la \Delta_k \partial_t B, \Delta_k B\ra+\sigma
	  \|\Delta_k B\|_{L^2}^2)-\tfrac{2}{c^2}\|\Delta_k\partial_t B\|_{L^2}^2+\|\Delta_k\na B\|_{L^2}^2\leq\sigma^2\| \Delta_k \tilde{J}\|_{L^2}^2\label{further-3}.
	\end{align}
    
We then multiply \eqref{further-1}, \eqref{further-2} and \eqref{further-3}  by $2^{2ks}$, $2^{2k(s-1)}$ and  $\eta_*2^{2k(s-1)}$  respectively,  sum over $k\leq k_0$, and add these three inequalities together. $\eta_*>0$ is a  sufficiently  small constant to be specified later. We obtain
\begin{align}\label{eb}
\tfrac{\d}{\d t}F(t)+G(t)\leq R(t).
\end{align}
Here $F$ is the energy function
\begin{align*}
F(t)=\tfrac{1}{c}\| E\|_{H^{s}_{\leq k_0}}^2&+\tfrac{1}{c}\| B\|_{H^{s}_{\leq k_0}}^2+\tfrac{1}{c^2}\|\partial_t B\|_{H^{s-1}_{\leq k_0}}^2+\|\na B\|_{H^{s-1}_{\leq k_0}}^2
 \\&+\sum_{k\leq k_0}\tfrac{2\eta_*}{c^2}2^{2k(s-1)}\la \Delta_k \partial_t B, \Delta_k B\ra+\eta_*\sigma
	  \|B\|_{H^{s-1}_{\leq k_0}}^2,
\end{align*}
$G$ is the dissipative function
$$G(t)=\sigma c\| E\|_{H^{s}_{\leq k_0}}^2
	+(\sigma-\tfrac{2\eta_*}{c^2}) \|\partial_t B\|_{H^{s-1}_{\leq k_0}}^2+\eta_*\|\na B\|_{H^{s-1}_{\leq k_0}}^2,$$
 and   the right-hand side $R(t)$ is expressed by:
\begin{align*}
R(t)=& \frac{\sigma}{c}
\| \tilde{J}\|_{H^{s}_{\leq k_0}}^2+
\sigma\|\na  \tilde{J}\|_{H^{s-1}_{\leq k_0}}^2
+\eta_*\sigma^2\|\tilde{J}\|_{H^{s-1}_{\leq k_0}}^2 .
\end{align*}

For $F$, using H\"{o}lder's inequality and Young's inequality, we obtain
\begin{align*}
 |\sum_{k\leq k_0}\tfrac{2\eta_*}{c^2}2^{2k(s-1)}\la \Delta_k \partial_t B, \Delta_k B\ra|\leq \tfrac{2\eta_*}{\sigma c^4}\|\partial_t B\|_{H^{s-1}_{\leq k_0}}^2+\tfrac{\eta_*\sigma}{2}\|B\|_{H^{s-1}_{\leq k_0}}^2.
\end{align*}
We choose $\eta_*$ sufficiently small such that
$\tfrac{2\eta_*}{\sigma c^4}\leq \tfrac{1}{2c^2}$. Then  $F$ is positive. Furthermore,  by virtue of \eqref{euqNSM}-(4), we observe that
\begin{align}\label{F-new}
F(t)\thickapprox
\| E\|_{H^{s}_{\leq k_0}}^2+\| B\|_{H^{s}_{\leq k_0}}^2.
\end{align}
For $G$, if
$\eta_*$ is chosen sufficiently small such that $(\sigma-\tfrac{2\eta_*}{c^2})>0$, using \eqref{euqNSM}-(4) again, we have
\begin{align}\label{G-new}
G(t)\thickapprox \| E\|_{H^{s}_{\leq k_0}}^2
	+\|\na B\|_{H^{s-1}_{\leq k_0}}^2.
    \end{align}
Finally, direct calculation gives 
\begin{align}\label{R-new}
R(t)\lesssim \|\tilde{J}\|_{H^{s}_{\leq k_0}}^2 .
  \end{align}
  
  Returning to \eqref{eb}, by performing time integration, we obtain that for any $T>0$,
 \begin{align}\label{FG}
   \int_0^{T} G(t)\d t\leq F(0)-F(T)+\int_0^{T} R(t)\d t\leq F(0)+\int_0^{T} R(t)\d t.
 \end{align}
 Substituting \eqref{F-new}-\eqref{R-new} into \eqref{FG},  we get that, for any $T>0$,  \begin{align}\label{ebebnew}
   \int_0^{T} (\| E\|_{H^{s}_{\leq k_0}}^2
	+\|\na B\|_{H^{s-1}_{\leq k_0}}^2)\d t\lesssim \| E^{in}\|_{H^{s}_{\leq k_0}}^2+\| B^{in}\|_{H^{s}_{\leq k_0}}^2 +\int_0^{T}\|\tilde{J}\|_{H^{s}_{\leq k_0}}^2 \d t.
 \end{align}
Recall the expression of $\tilde{J}$:
$$\tilde{J}=\varepsilon^2\partial_tj-\varepsilon^2\mu\Delta j-u\times B+\varepsilon^2u\cdot \nabla j+\varepsilon^2j\cdot \nabla u.$$
Using \eqref{ffk0-new} and \eqref{l+}, it is straightforward to see that
\begin{align}\label{haofan}
\|\varepsilon^2\mu\Delta& j \|_{H^{s}_{\leq k_0}}^2
 \lesssim \varepsilon^2\|\varepsilon\na j\|_{H^{s+1}_{\leq k_0}}^2
 \lesssim \varepsilon^22^{2k_0}\|\varepsilon\na j\|_{H^{s}}^2.
\end{align}
Using \eqref{ffk0-new} and \eqref{l+} again, together with  Lemma \ref{jb}, we have
\begin{align*}
 \|u\times B\|_{H^{s}_{\leq k_0}}^2\lesssim
  \|u\times B\|_{H^{s}}^2\lesssim  \|\nabla u\|_{H^{s'}}^2\|B\|_{H^{s}}^2,
  \end{align*}
\begin{align}\label{e2uj-new}
 \|\varepsilon^2u\cdot \nabla j\|_{H^{s}_{\leq k_0}}^2\lesssim  \varepsilon^22^{2k_0}
   \| u\cdot \nabla \varepsilon j\|_{H^{s-1}}^2
 \lesssim  \varepsilon^22^{2k_0}\| u\|_{H^{s'}}^2\|\varepsilon \na j\|_{H^{s}}^2,
\end{align}
and
\begin{align}\label{e2uj}
\|\varepsilon^2 j\cdot \nabla u\|_{H^{s}_{\leq k_0}}^2\lesssim  \varepsilon^22^{2k_0}
  \|\varepsilon j\cdot \nabla u\|_{H^{s-1}}^2
 \lesssim \varepsilon^22^{2k_0}\|\nabla u\|_{H^{s'}}^2\|\varepsilon j\|_{H^{s}}^2.
\end{align}
By virtue of the energy estimate \eqref{uni-2},
we we deduce that 
\begin{align*}
\int_0^T  (\|\varepsilon^2\mu\Delta j \|_{H^{s}_{\leq k_0}}^2+\|u\times B\|_{H^{s}_{\leq k_0}}^2+\|\varepsilon^2u\cdot \nabla j\|_{H^{s}_{\leq k_0}}^2+ \|\varepsilon^2 &j\cdot \nabla u\|_{H^{s}_{\leq k_0}}^2)\d t\\
\lesssim &
(\kappa_0+\kappa_0^2)(1+\varepsilon^22^{2k_0}).
\end{align*}
Plugging the above inequality into  \eqref{ebebnew} gives \begin{align*}
   \int_0^{T} (\| E\|_{H^{s}_{\leq k_0}}^2
	+\|\na B\|_{H^{s-1}_{\leq k_0}}^2)\d t\lesssim &(\kappa_0+\kappa_0^2)(1+\varepsilon^22^{2k_0})+\int_0^{T}\|\varepsilon^2\partial_tj\|_{H^{s}_{\leq k_0}}^2 \d t.
 \end{align*}
 Since $0<\varepsilon\leq 1$ and $k_0\geq -1$,  we have $(1+\varepsilon^22^{2k_0})\lesssim 2^{2k_0}$.
 We ultimately obtian that
\begin{align}\label{EB}
   \|\na B\|_{L^2_TH^{s-1}_{\leq k_0}}\leq C(2^{k_0}+\|\varepsilon^2\partial_tj\|_{L^2_TH^{s}_{\leq k_0}}).
 \end{align}
 We mention that $C$ is independent of $T$. 
 
{\bf {Step 2: estimates for $\varepsilon^2\partial_tj$.}}
Applying $\Delta_k$ to system \eqref{euqNSM}-(2) and taking the $L^2$-inner product with $ \varepsilon^2 \Delta_k\partial_t j$, we have
\begin{align}\label{further}
\tfrac{1}{2}\tfrac{\d}{\d t}(\mu\varepsilon^2\|\Delta_k\varepsilon\na j\|_{L^2}^2
&+\tfrac{1}{\sigma}\|\Delta_k\varepsilon j\|_{L^2}^2)+
\|\Delta_k\varepsilon^2\partial_t j\|_{L^2}^2\no\\&=
 \langle\Delta_k (cE+u\times B-\varepsilon^2u\cdot \nabla j-\varepsilon^2j\cdot \nabla u), \Delta_k\varepsilon^2\partial_t j\rangle_{L^2}.
\end{align}
Performing  integration by parts gives
\begin{align*}
  \langle \Delta_kE, \Delta_k\varepsilon^2\partial_t j\ra_{L^2}= & \varepsilon^2
 \tfrac{\d}{\d t}\langle \Delta_kE, \Delta_k  j\ra_{L^2}-\varepsilon^2\langle \Delta_k\partial_tE, \Delta_k  j\ra_{L^2}.
 \end{align*}
 Using \eqref{euqNSM}-(3), we see that 
 \begin{align*}
 \langle \Delta_k\partial_tE, \Delta_k  j\ra_{L^2} 
 =c\langle \Delta_k \nabla\times B, \Delta_k
  j\ra_{L^2}
  -c\|\Delta_k j\|_{L^2}^2.
 \end{align*}
This leads to
 \begin{align*}
  \langle \Delta_kE, \Delta_k\varepsilon^2\partial_t j\ra_{L^2}= \varepsilon 
 \tfrac{\d}{\d t}\langle \Delta_kE, \Delta_k  \varepsilon j\ra_{L^2}-c\varepsilon^2\langle \Delta_k \nabla\times B, \Delta_k
  j\ra_{L^2}
+c\varepsilon^2\|\Delta_k j\|_{L^2}^2.
 \end{align*}
Similarly, performing  integration by parts, we have
\begin{align*}
  \langle \Delta_k(u\times B), \Delta_k\varepsilon^2\partial_t j\ra_{L^2}= & \varepsilon
 \tfrac{\d}{\d t}\langle \Delta_k( u\times B), \Delta_k \varepsilon j\ra_{L^2}\\&-\varepsilon^2\langle \Delta_k(\partial_tu \times B), \Delta_k j\ra_{L^2}
 -\varepsilon^2\langle \Delta_k(u \times \partial_t B), \Delta_k j\ra_{L^2}.\end{align*}
The above process ensures that $\varepsilon$ appears in the coefficients. Finally, using H\"{o}lder's inequality and Young's inequality, we obtain
\begin{align*}
   -\langle\Delta_k \varepsilon^2(u\cdot \nabla j), \Delta_k\varepsilon^2\partial_t j\rangle_{L^2}\leq
   \varepsilon^2\|\Delta_k (u\cdot \nabla \varepsilon j)\|_{L^2}^2+\tfrac{1}{4}\|\Delta_k\varepsilon^2\partial_t j\|_{L^2}^2,
 \end{align*}
 \begin{align*}
   -\langle\Delta_k \varepsilon^2(j\cdot \nabla u), \Delta_k\varepsilon^2\partial_t j\rangle_{L^2}\leq
\varepsilon^2\|\Delta_k (\varepsilon j\cdot \nabla u)\|_{L^2}^2+\tfrac{1}{4}\|\Delta_k\varepsilon^2\partial_t j\|_{L^2}^2.
 \end{align*}
 Plugging the above estimates into \eqref{further}, we get
\begin{align}\label{e-nab}
\tfrac{\d}{\d t}(\mu\varepsilon^2\|\Delta_k\varepsilon\na j\|_{L^2}^2
&+\tfrac{1}{\sigma}\|\Delta_k\varepsilon j\|_{L^2}^2
-2\varepsilon\langle \Delta_k cE, \Delta_k \varepsilon j\ra_{L^2}\no\\&-2\varepsilon\langle \Delta_k(u\times B), \Delta_k \varepsilon j\ra_{L^2})+
\|\Delta_k\varepsilon^2\partial_t j\|_{L^2}^2\leq \sum_{i=1}^6{\tilde I}_i,
\end{align}
where 
\begin{align*}
\tilde{I}_1=-2c^2\varepsilon^2\langle \Delta_k\nabla\times B,\Delta_k j\ra_{L^2}, ~
\tilde{I}_2  =2c^2\varepsilon^2\|\Delta_k j\|_{L^2}^2, ~\tilde{I}_3=-2\varepsilon^2\langle \Delta_k(\partial_tu \times B), \Delta_k j\ra_{L^2},\end{align*}
\begin{align*}
 \tilde{I}_4= -2\varepsilon^2\langle \Delta_k(u \times \partial_t B), \Delta_k j\ra_{L^2}, ~\tilde{I}_5 =2\varepsilon^2\|\Delta_k(u\cdot \nabla \varepsilon j
 )\|_{L^2}^2, ~
\tilde{I}_6=2\varepsilon^2\|\Delta_k(\varepsilon j\cdot \nabla u)\|_{L^2}^2.
\end{align*}

Multiplying both sides of \eqref{e-nab} by $2^{2ks}$ and  summing over $k\leq k_0$, we have
\begin{align*}
\tfrac{\d}{\d t}\tilde{F}(t)+
\|\varepsilon^2\partial_t j\|_{H^{s}_{\leq k_0}}^2\leq \tilde{R}(t),
\end{align*}
where
\begin{align*}
  \tilde{F}(t)=\mu\varepsilon^2\|\varepsilon\na j\|_{H^{s}_{\leq k_0}}^2
+\tfrac{1}{\sigma}\|\varepsilon j\|_{H^{s}_{\leq k_0}}^2
&-2\varepsilon\sum_{k\leq  k_0}2^{2ks}\langle \Delta_k cE, \Delta_k \varepsilon j\ra_{L^2}\\&-2\varepsilon\sum_{k\leq  k_0}2^{2ks}\langle \Delta_k(u\times B), \Delta_k \varepsilon j\ra_{L^2}, \end{align*}
and \begin{align*}
R(t)=\sum_{i=1}^6(\sum_{k\leq  k_0}2^{2ks}  {\tilde I}_i)\stackrel{\triangle}{=}\sum_{i=1}^6\tilde{R}_i(t).
\end{align*}
 Integrating in time, we have, for any $T>0$,
\begin{align}\label{e2j0}
  \int_0^{T}\|\varepsilon^2\partial_t j\|_{H^{s}_{\leq k_0}}^2\d t
 \leq (\tilde{F}(0)-\tilde{F}(T))+\int_0^{T}\tilde{R}(t)\d t.
\end{align}

For $F$, direct calculation gives
 \begin{align}\label{tiledF}
   \tilde{F}(0)-\tilde{F}(T)
    \lesssim &\varepsilon^2
    \|\varepsilon j^{in}\|_{H^{s+1}_{\leq k_0}}^2+ \|\varepsilon j^{in}\|_{H^{s}_{\leq k_0}}^2\no\\&
    +\varepsilon\sup_{t\geq 0}(\|E\|_{H^s_{\leq k_0}}\|\varepsilon j\|_{H^s_{\leq k_0}}+
   \|u\times B\|_{H^{s}_{\leq k_0}}\|\varepsilon j\|_{H^s_{\leq k_0}}).
   \end{align}
   By virtue of \eqref{ffk0-new}, \eqref{l+} and  Lemma ~\ref{jb},  we have
\begin{align*}
 & \varepsilon^2\|\varepsilon j^{in}\|_{H^{s+1}_{\leq k_0}}^2\lesssim \varepsilon^22^{2k_0}\|\varepsilon j^{in}\|_{H^s}^2,\\
 &\|u\times B\|_{H^{s}_{\leq k_0}}\lesssim 2^{k_0}\|u\times B\|_{H^{s-1}}
\lesssim 2^{k_0}\|u\|_{H^{s'}}\|B\|_{H^{s}}.
\end{align*} 
  Using  \eqref{ffk0-new} again, we deduce  from \eqref{tiledF} that:
 \begin{align}\label{Ftilde-new}
   \tilde{F}(0)-\tilde{F}(T)
    \lesssim
    &\varepsilon^22^{2k_0}\|\varepsilon j^{in}\|_{H^s}^2+ \|\varepsilon j^{in}\|_{H^{s}_{\leq k_0}}^2\no\\&
    +\varepsilon\sup_{t\geq 0}(\|E\|_{H^s}\|\varepsilon j\|_{H^s}+
2^{k_0}\|u\|_{H^{s'}}\|B\|_{H^{s}}\|\varepsilon j\|_{H^s}).\end{align}
Note that  $\varepsilon$ appears in the first  and third terms on the right-hand side of \eqref{Ftilde-new}. Then, using  the energy bound \eqref{uni-2}, along with 
$0<\varepsilon\leq 1$ and $k_0\geq -1$, we have 
\begin{align*}
\varepsilon^22^{2k_0}\|\varepsilon j^{in}\|_{H^s}^2
\lesssim \varepsilon^22^{2k_0} \kappa_0 \lesssim \varepsilon2^{2k_0},
\end{align*}
and 
\begin{align*}
\varepsilon\sup_{t\geq 0}(\|E\|_{H^s}\|\varepsilon j\|_{H^s}+
2^{k_0}\|u\|_{H^{s'}}\|B\|_{H^{s}}\|\varepsilon j\|_{H^s})\lesssim \varepsilon (\kappa_0+2^{k_0}\kappa_0^{\tfrac{3}{2}})\lesssim \varepsilon 2^{2k_0}.
\end{align*}
Plugging the above two inequalities into \eqref{Ftilde-new}, finally, we have 
 \begin{align}\label{Ftilde}
   \tilde{F}(0)-\tilde{F}(T)
    \lesssim &\varepsilon2^{2k_0}+\|\varepsilon j^{in}\|_{H^{s}}^2.
    \end{align}
    
We now treat $\tilde{R}$. We will repeatedly use \eqref{ffk0-new} and \eqref{l+}. For $\tilde{R}_1$, direct calculation gives
\begin{align*}
  &{\tilde R}_1\lesssim  \varepsilon^2\|\nabla B\|_{H^{s-1}_{\leq k_0}}\| j\|_{H^{s+1}_{\leq k_0}}
   \lesssim \varepsilon^22^{k_0}\|\nabla B\|_{H^{s-1}_{\leq k_0}}\| j\|_{H^{s}}.\end{align*}
  This leads to 
   \begin{align*}
  \int_0^T {\tilde R}_1\d t\lesssim \varepsilon^22^{k_0}\|\nabla B\|_{L^2_TH^{s-1}_{\leq k_0}}\| j\|_{L^2_TH^{s}}.
  \end{align*}
By \eqref{EB} obtained in Step 1 and by applying Young's inequality, we infer that
\begin{align*}
  \int_0^T {\tilde R}_1\d t\lesssim &
  \varepsilon^22^{k_0}(2^ {k_0}+\|\varepsilon^2\partial_tj\|_{L^2_TH^{s}_{\leq k_0}})\| j\|_{L^2_TH^{s}} \\
  \leq &C
  \varepsilon^22^{2k_0}\| j\|_{L^2_TH^{s}}+C\varepsilon^42^{2k_0}\| j\|_{L^2_TH^{s}}^2+\tfrac{1}{2}\|\varepsilon^2\partial_tj\|_{L^2_TH^{s}_{\leq k_0}}^2.
 \end{align*}     
Next, it is straightforward to see 
${\tilde R}_2\lesssim \varepsilon^2\| j\|_{H^{s}}^2.$ As for $\tilde{R}_3$,  we have 
\begin{align*}
 & {\tilde R}_3\lesssim  
   \varepsilon^2\|\partial_tu \times  B\|_{H^{s}_{\leq k_0}}\|j\|_{H^{s}_{\leq k_0}}.
   \end{align*}
According to Bony's decomposition \eqref{Bony}, along with Lemma \ref{jb} and the condition $s-1\leq s'\leq s+1$, we get
\begin{align*}
  \|\partial_tu \times  B\|_{H^{s}_{\leq k_0}}
\lesssim & 2^{2k_0}
 \|T_{\partial_tu} B\|_{H^{s-2}}+2^{k_0(s-s'+1)}\|T_B{\partial_tu}+R(\partial_t u,B)\|_{H^{s’-1}}\\
  \lesssim &2^{2k_0}
 \|\partial_tu\|_{H^{s'-1}} \|B\|_{H^{s}}.
  \end{align*}
This implies that
\begin{align*}
  {\tilde R}_3
  \lesssim &\varepsilon^22^{2k_0}
 \|\partial_tu\|_{H^{s'-1}} \|B\|_{H^{s}}\|j\|_{H^{s}}.
  \end{align*}
  Similarly, using Lemma \ref{jb}  and \eqref{euqNSM}-(4), we have
\begin{align*}
  {\tilde R}_4\lesssim &
  \varepsilon^2\|u \times \partial_t B\|_{H^{s-1}_{\leq k_0}}\|j\|_{H^{s+1}_{\leq k_0}}\\
  \lesssim &\varepsilon^22^{k_0}\|\nabla u\|_{H^{s'}}\|\partial_t B\|_{H^{s-1}}\|j\|_{H^{s}}
   \lesssim \varepsilon^22^{k_0}\|\nabla u\|_{H^{s'}}\|E\|_{H^{s}}\|j\|_{H^{s}}.
\end{align*}
Finally, thanks to \eqref{e2uj-new} and \eqref{e2uj}, we have \begin{align*}
 {\tilde R}_5+ {\tilde R}_6\lesssim &\varepsilon^2(\|u\cdot \nabla \varepsilon j\|_{H^{s}_{\leq k_0}}^2+\|\varepsilon j\cdot \nabla u\|_{H^{s}_{\leq k_0}}^2)\\
 \lesssim & \varepsilon^22^{2k_0}(\| u\|_{H^{s'}}^2\|\varepsilon \na j\|_{H^{s}}^2+\|\nabla u\|_{H^{s'}}^2\|\varepsilon j\|_{H^{s}}^2).
\end{align*}
Then, we deduce from  \eqref{uni-2} (this also yields $\partial_tu \in L^2(\mathbb{R}^+;H^{s'-1})$ and $\|\partial_tu\|_{L^2(\mathbb{R}^+;H^{s'-1})}\leq C(\kappa_0))$, $0<\varepsilon\leq 1$ and $k_0\geq -1$ that
 \begin{align*}
   \int_0^{T}\tilde{R}(t)\d t\leq C \varepsilon^22^{2k_0}+\tfrac{1}{2}\|\varepsilon\partial_tj\|_{L^2_TH^{s}_{\leq k_0}}^2.
 \end{align*}
Substituting the above estimate and \eqref{Ftilde}  into \eqref{e2j0}, we get
\begin{align*}
  \int_0^{T}\|\varepsilon^2\partial_t j\|_{H^{s}_{\leq k_0}}^2\d t
 \leq C(\varepsilon2^{2k_0}+\|\varepsilon j^{in}\|_{H^{s}}^2)
\end{align*} holds true for any $T>0$. We mention that $C$ is independent of $T$.  Let $T\rightarrow +\infty$ then completes the proof of the lemma.\end{proof}

\section{Strong convergence results}
\subsection{Estimates for the difference between two solutions }
Let $\varepsilon_1$ and $\varepsilon_2$ be two momentum transfer coefficients,  $0< \varepsilon_1,\varepsilon_2\leq 1$.  Suppose that the initial date satisfy
 \begin{align*}	  \|u^{\varepsilon_i,in}\|_{H^{s'}(\mathbb{R}^3)}^2 + \| \varepsilon_ij^{\varepsilon_i,in}\|_{H^s(\mathbb{R}^3)}^2
		  + \|E^{\varepsilon_i,in}\|_{H^{s}(\mathbb{R}^3)}^2
		  + \|B^{\varepsilon_i,in}\|_{H^{s}(\mathbb{R}^3)}^2
			\leq \kappa_0,\, i=1,2,
	  \end{align*}
for some sufficiently small $\kappa_0$.  Let $(u^{\varepsilon_i},j^{\varepsilon_i},E^{\varepsilon_i},B^{\varepsilon_i})$, $i=1,2,$ be the corresponding global solutions.
Denote
\begin{align*}
&\mathcal{E}^{\varepsilon_2,\varepsilon_1}(t)=\|u^{\varepsilon_2}-u^{\varepsilon_1}\|^2_{H^{s'}}
 +\|E^{\varepsilon_2}-E^{\varepsilon_1}\|^2_{H^{s}}
 +\|B^{\varepsilon_2}-B^{\varepsilon_1}\|^2_{H^{s}},\\
 &\mathcal{D}^{\varepsilon_2,\varepsilon_1}(t)=\mu\|\na( u^{\varepsilon_2}-u^{\varepsilon_1})\|^2_{H^{s'}}
 +\tfrac{1}{\sigma}\|j^{\varepsilon_2}-j^{\varepsilon_1}\|^2_{H^{s}},
 \end{align*}
$\mathcal{E}^{\varepsilon_i}_{\leq k_0}$, $\mathcal{E}^{\varepsilon_i}_{> k_0}$, 
$\mathcal{D}^{\varepsilon_i}_{\leq k_0}$, $\mathcal{D}^{\varepsilon_i}_{> k_0}$, $i=1,2,$ $\mathcal{E}^{\varepsilon_2,\varepsilon_1}_{\leq k_0}$, $\mathcal{E}^{\varepsilon_2,\varepsilon_1}_{> k_0},$ $\mathcal{D}^{\varepsilon_2,\varepsilon_1}_{\leq k_0},$ $\mathcal{D}^{\varepsilon_2,\varepsilon_1}_{> k_0}$ are the corresponding low and high frequency truncation functions.

\begin{lemma}\label{low-diff}
For any $k_0\geq -1$, there exists a positive constant $C$, independent of $\varepsilon_1,\varepsilon_2$ and $k_0$, such that
 \begin{align*}
 \sup_{t\geq 0}\mathcal{E}^{\varepsilon_2,\varepsilon_1}_{\leq k_0}(t)+&\int_0^{+\infty}\mathcal{D}^{\varepsilon_2,\varepsilon_1}_{\leq k_0}(t)\d t
 \leq C\Big( \mathcal{E}^{\varepsilon_2,\varepsilon_1,in}+\|\varepsilon_1 j^{\varepsilon_1,in}\|_{H^{s}}^2
    +\|\varepsilon_2 j^{\varepsilon_2,in}\|_{H^{s}}^2
  \\&  +\sup_{t\geq 0}(\mathcal{E}^{\varepsilon_1}_{>k_0}(t)
    +\mathcal{E}^{\varepsilon_2}_{>k_0}(t))
+\int_0^{+\infty}
(\mathcal{D}^{\varepsilon_1}_{>k_0}(t)
+\mathcal{D}^{\varepsilon_2}_{>k_0}(t))\d t
+(\varepsilon_1+\varepsilon_2)2^{4k_0}
\Big).
\end{align*}
 \end{lemma}
 \begin{proof}
The difference
$(u^{\varepsilon_2}-u^{\varepsilon_1},
j^{\varepsilon_2}-j^{\varepsilon_1},E^{\varepsilon_2}-E^{\varepsilon_1},
B^{\varepsilon_2}-B^{\varepsilon_1})$ reads
\begin{align*}
  \left\{
\begin{array}{ll}
\partial_t(u^{\varepsilon_2}-u^{\varepsilon_1})-\mu\Delta (u^{\varepsilon_2}-u^{\varepsilon_1})+\nabla (p^{\varepsilon_2}-p^{\varepsilon_1})=II_1+II_2+II_3+II_4+III_1+III_2,\\[1ex]
\tfrac{1}{\sigma}(j^{\varepsilon_2}-j^{\varepsilon_1})+\nabla ( \tilde{p}^{\varepsilon_2}-\tilde{p}^{\varepsilon_1})-c(E^{\varepsilon_2}-
E^{\varepsilon_1})=II_5+II_6+III_3+III_4, \\[1ex]
\tfrac{1}{c}\partial_t(E^{\varepsilon_2}-E^{\varepsilon_1})-\nabla \times (B^{\varepsilon_2}-B^{\varepsilon_1})+(j^{\varepsilon_2}-j^{\varepsilon_1})=0, \\[1ex]
\tfrac{1}{c}\partial_t(B^{\varepsilon_2}-B^{\varepsilon_1})+\nabla \times (E^{\varepsilon_2}-E^{\varepsilon_1})=0, \\[1ex]
\div (u^{\varepsilon_2}-u^{\varepsilon_1})=0,\ \div (j^{\varepsilon_2}-j^{\varepsilon_1})=0,\ \div (E^{\varepsilon_2}-E^{\varepsilon_1})=0,\ \div (B^{\varepsilon_2}-B^{\varepsilon_1})=0,
\end{array}
\right.
\end{align*}
where
\begin{align*}
  &II_1=(j^{\varepsilon_2}-j^{\varepsilon_1})\times B^{\varepsilon_2},\,\,
  II_2=j^{\varepsilon_1}\times (B^{\varepsilon_2}-B^{\varepsilon_1}),\,\,
  II_3=-( u^{\varepsilon_2}-u^{\varepsilon_1})\cdot \nabla u^{\varepsilon_2},\\
  &II_4=-u^{\varepsilon_1}\cdot\nabla (u^{\varepsilon_2}-u^{\varepsilon_1}),\,\,
II_5=(u^{\varepsilon_2}-u^{\varepsilon_1})\times B^{\varepsilon_2},\,\,
II_6=u^{\varepsilon_1}\times (B^{\varepsilon_2}-B^{\varepsilon_1}),\\
&III_1=-
\varepsilon_2^2 j^{\varepsilon_2}\cdot \na j^{\varepsilon_2},\,\,
III_2=\varepsilon_1^2 j^{\varepsilon_1}\cdot \na j^{\varepsilon_1},\\
&III_3=-\varepsilon_2^2\partial_tj^{\varepsilon_2}
-\varepsilon_2^2u^{\varepsilon_2}\cdot \nabla j^{\varepsilon_2}
-\varepsilon_2^2j^{\varepsilon_2}\cdot \nabla u^{\varepsilon_2}
+\varepsilon_2^2\mu\Delta j^{\varepsilon_2},\\
&III_4=\varepsilon_1^2\partial_tj^{\varepsilon_1}
+\varepsilon_1^2u^{\varepsilon_1}\cdot \nabla j^{\varepsilon_1}
+\varepsilon_1^2j^{\varepsilon_1}\cdot \nabla u^{\varepsilon_1}
-\varepsilon_1^2\mu\Delta j^{\varepsilon_1}.
\end{align*}
Following along the same lines as the proof of  \eqref{us'}, summing over $ -1\leq k\leq k_0$ instead of summing over $ -1\leq k< +\infty$,  we get
\begin{align}\label{diff-jeb}
 \tfrac{1}{2}\tfrac{\d}{\d t}&\mathcal{E}^{\varepsilon_2,\varepsilon_1}_{\leq k_0}+\mathcal{D}^{\varepsilon_2,\varepsilon_1}_{\leq k_0}
 =
 \sum_{k\leq k_0}2^{2ks'}\la\Delta_k\big(\sum_{m=1}^{4}II_m+\sum_{m=1}^{2}III_m\big),\Delta_k (u^{\varepsilon_2}-u^{\varepsilon_1})\ra_{L^2}\no\\&+
 \sum_{k\leq k_0}2^{2ks}\la\Delta_k\big(\sum_{m=5}^6II_m+\sum\limits_{m=3}^{4}III_m
\big),\Delta_k (j^{\varepsilon_2}-j^{\varepsilon_1})\ra_{L^2}\triangleq \sum_{m=1}^6R_{II_m}+\sum_{m=1}^4R_{III_m}.
\end{align}

By adapting the method used to control $I_{1,k}$ and $I_{2,k}$ in Lemma \ref {lemm:apriori}, using the inequality \eqref{ffk0-new} and Lemma \ref{jb}, we can obtain 
\begin{align*}
  &R_{II_2}
  \lesssim \|j^{\varepsilon_1}\|_{H^s}\|B^{\varepsilon_2}-B^{\varepsilon_1}\|_{H^{s}}
 \| \na (u^{\varepsilon_2}-u^{\varepsilon_1})\|_{H^{s'}},\\
  &R_{II_3}
  \lesssim \|u^{\varepsilon_2}-u^{\varepsilon_1}\|_{H^{s'}}
  \|\na u^{\varepsilon_2}\|_{H^{s'}}
 \| \na (u^{\varepsilon_2}-u^{\varepsilon_1})\|_{H^{s'}},\\
 &R_{II_6}
  \lesssim \|\na u^{\varepsilon_1}\|_{H^{s'}}
  \|B^{\varepsilon_2}-B^{\varepsilon_1}\|_{H^{s}}
 \| j^{\varepsilon_2}-j^{\varepsilon_1}\|_{H^{s}}.
 \end{align*}
This gives 
 \begin{align*}
 R_{II_2}+R_{II_3}+R_{II_6}\lesssim &\big((\mathcal{D}^{\varepsilon_1})^{\fab}+
 (\mathcal{D}^{\varepsilon_2})^{\fab}\big)(\mathcal{E}^{\varepsilon_2,\varepsilon_1})^{\fab}
 (\mathcal{D}^{\varepsilon_2,\varepsilon_1})^{\fab}\\
 \leq &C(\mathcal{D}^{\varepsilon_1}+
 \mathcal{D}^{\varepsilon_2})\mathcal{E}^{\varepsilon_2,\varepsilon_1}+\tfrac{1}{8}
 \mathcal{D}^{\varepsilon_2,\varepsilon_1}.
 \end{align*}
Similarly, we have
 \begin{align*}
  &R_{II_1}
  \lesssim \|j^{\varepsilon_2}-j^{\varepsilon_1}\|_{H^s}\|B^{\varepsilon_2}\|_{H^{s}}
 \| \na (u^{\varepsilon_2}-u^{\varepsilon_1})\|_{H^{s'}},\\
  &R_{II_4}
  \lesssim \|u^{\varepsilon_1}\|_{H^{s'}}
  \|\na (u^{\varepsilon_2}-u^{\varepsilon_1})\|_{H^{s'}}
 \| \na (u^{\varepsilon_2}-u^{\varepsilon_1})\|_{H^{s'}},\\
 &R_{II_5}
  \lesssim \|\na (u^{\varepsilon_2}-u^{\varepsilon_1})\|_{H^{s'}}
  \|B^{\varepsilon_2}\|_{H^{s}}
 \| j^{\varepsilon_2}-j^{\varepsilon_1}\|_{H^{s}}.
\end{align*}
Recall $\kappa_0$ is sufficient small. We deduce from \eqref{uni-2} that
\begin{align*}
 R_{II_1}+R_{II_4}+R_{II_5}\leq C\big( (\mathcal{E}^{\varepsilon_1})^{\fab}+
 (\mathcal{E}^{\varepsilon_2})^{\fab}\big)\mathcal{D}^{\varepsilon_2,\varepsilon_1} \lesssim \kappa_0^{\fab} \mathcal{D}^{\varepsilon_2,\varepsilon_1}\leq \tfrac{1}{2}
\mathcal{D}^{\varepsilon_2,\varepsilon_1}.
 \end{align*}
As for $R_{III_1}$,  direct  calculation gives
\begin{align*}
  R_{III_1}
  \lesssim &\varepsilon_2^2\|j^{\varepsilon_2}\cdot \na j^{\varepsilon_2}\|_{H^{s'}_{k\leq k_0}}\| u^{\varepsilon_2}-u^{\varepsilon_1}\|_{H^{s'}_{k\leq k_0}}
  \end{align*}
  Using  \eqref{ffk0-new} and \eqref{l+}, $s-1 \leq s'\leq s+1$ and Lemma \ref{jb}, we have
\begin{align}\label{youyong}
\|j^{\varepsilon_2}\cdot \na j^{\varepsilon_2}\|_{H^{s'}_{k\leq k_0}}\lesssim 2^{k_0(s'-s+1)}\|j^{\varepsilon_2}\cdot \na j^{\varepsilon_2}\|_{H^{s-1}}\lesssim 2^{2k_0}\|j^{\varepsilon_2}\|_{H^{s}}^2.
\end{align}
 Together with Young's inequality, we obtain
 \begin{align*}
  R_{III_1}
  \lesssim &\varepsilon_2^22^{2k_0}
  \|j^{\varepsilon_2}\|_{H^{s}}^2\| u^{\varepsilon_2}-u^{\varepsilon_1}\|_{H^{s'}}\lesssim \varepsilon_2^42^{4k_0}
  \|j^{\varepsilon_2}\|_{H^{s}}^2+  \|j^{\varepsilon_2}\|_{H^{s}}^2\| u^{\varepsilon_2}-u^{\varepsilon_1}\|_{H^{s'}}^2.
  \end{align*}
 Similarly, we have
  \begin{align*}
 R_{III_2}
 \lesssim \varepsilon_1^42^{4k_0}
  \|j^{\varepsilon_1}\|_{H^{s}}^2+  \|j^{\varepsilon_1}\|_{H^{s}}^2\| u^{\varepsilon_2}-u^{\varepsilon_1}\|_{H^{s'}}^2.
  \end{align*}
Next, direct calculation gives
  \begin{align*}
    R_{III_3}
     \lesssim &\|III_3\|_{H^{s}_{k\leq k_0}}
     \| j^{\varepsilon_2}-j^{\varepsilon_1}\|_{H^{s}_{k\leq k_0}}
     \leq C\|III_3\|_{H^{s}_{k\leq k_0}}^2+\tfrac{1}{16}\| j^{\varepsilon_2}-j^{\varepsilon_1}\|_{H^s}^2.
     \end{align*}
   Similar arguments lead to  \begin{align*}
    R_{III_4}
     \leq C\|III_4\|_{H^{s}_{k\leq k_0}}^2+\tfrac{1}{16}\| j^{\varepsilon_2}-j^{\varepsilon_1}\|_{H^s}^2.
     \end{align*}
  Finally,  by using \eqref{haofan}, \eqref{e2uj-new} and \eqref{e2uj}, we have
      \begin{align*}
   \|III_3\|_{H^{s}_{k\leq k_0}}^2+\|III_4\|_{H^{s}_{k\leq k_0}}^2\lesssim &\sum_{i=1}^2
     \big(\|\varepsilon_i^2\partial_tj^{\varepsilon_i}\|_{H^s_{k\leq k_0}}^2+
    \varepsilon_i^2 2^{2k_0}(\|\varepsilon_i\na j^{\varepsilon_i}\|_{H^{s}}^2\\&+
     \|\nabla u^{\varepsilon_i}\|_{H^{s'}}^2\|\varepsilon j^{\varepsilon_i}\|_{H^{s}}^2+\| u^{\varepsilon_i}\|_{H^{s'}}^2\|\varepsilon \na j^{\varepsilon_i}\|_{H^{s}}^2)\big).
     \end{align*}
  Since $0<\varepsilon_1,\varepsilon_2\leq 1$ and $k_0\geq -1$, we eventually get  that,
 \begin{align*}
 \sum_{m=1}^4R_{III_m}
 \leq &C\Big((\varepsilon_1^2+\varepsilon_2^2)2^{4k_0}(1+\mathcal{E}^{\varepsilon_1}+
 \mathcal{E}^{\varepsilon_2})
 (\mathcal{D}^{\varepsilon_1}+
 \mathcal{D}^{\varepsilon_2})
 \\&+(\mathcal{D}^{\varepsilon_1}+
 \mathcal{D}^{\varepsilon_2})\mathcal{E}^{\varepsilon_2,\varepsilon_1}
 +\|\varepsilon_1^2\partial_tj^{\varepsilon_1}\|_{H^s_{k\leq k_0}}^2+\|\varepsilon_2^2\partial_tj^{\varepsilon_2}\|_{H^s_{k\leq k_0}}^2\Big)
 +\tfrac{1}{8}
 \mathcal{D}^{\varepsilon_2,\varepsilon_1}.
 \end{align*}
 
Substituting the estimates for $R_{II}$ and $R_{III}$  into  \eqref{diff-jeb},  we have
\begin{align*}
 \tfrac{1}{2}\tfrac{\d}{\d t}\mathcal{E}^{\varepsilon_2,\varepsilon_1}_{\leq k_0}+\mathcal{D}^{\varepsilon_2,\varepsilon_1}_{\leq k_0}
 \leq &C\Big((\mathcal{D}^{\varepsilon_1}+
\mathcal{D}^{\varepsilon_2})\mathcal{E}^{\varepsilon_2,\varepsilon_1}
 + \|\varepsilon_1^2\partial_tj^{\varepsilon_1}\|_{H^s_{k\leq k_0}}^2+
\|\varepsilon_i^2\partial_tj^{\varepsilon_i}\|_{H^s_{k\leq k_0}}^2
\\& +(\varepsilon_1^2+\varepsilon_2^2)2^{4k_0}(1+\mathcal{E}^{\varepsilon_1}+
 \mathcal{E}^{\varepsilon_2})
 (\mathcal{D}^{\varepsilon_1}+
 \mathcal{D}^{\varepsilon_2})\Big)
+\tfrac{3}{4}\mathcal{D}^{\varepsilon_2,\varepsilon_1}.
 \end{align*}
 According to  the decomposition
 \begin{align}
\label{dec-new}\mathcal{E}^{\varepsilon_2,\varepsilon_1}=
   \mathcal{E}^{\varepsilon_2,\varepsilon_1}_{\leq k_0}+\mathcal{E}^{\varepsilon_2,\varepsilon_1}_{> k_0},\,\,
   \mathcal{D}^{\varepsilon_2,\varepsilon_1}=
   \mathcal{D}^{\varepsilon_2,\varepsilon_1}_{\leq k_0}+\mathcal{D}^{\varepsilon_2,\varepsilon_1}_{> k_0},
 \end{align}
 we finally conclude that
 \begin{align*}
 \tfrac{1}{2}\tfrac{\d}{\d t}\mathcal{E}^{\varepsilon_2,\varepsilon_1}_{\leq k_0}+\tfrac{1}{4}\mathcal{D}^{\varepsilon_2,\varepsilon_1}_{\leq k_0}
 \leq & C
 (\mathcal{D}^{\varepsilon_1}+
 \mathcal{D}^{\varepsilon_2})\mathcal{E}^{\varepsilon_2,\varepsilon_1}_{\leq k_0}
 +H(t),\end{align*}
 where \begin{align*}
 H(t)=&
 C\Big((\mathcal{D}^{\varepsilon_1}+
 \mathcal{D}^{\varepsilon_2})\mathcal{E}^{\varepsilon_2,\varepsilon_1}_{>k_0}
 +\|\varepsilon_1^2\partial_tj^{\varepsilon_1}\|_{H^s_{k\leq k_0}}^2+\|\varepsilon_2^2\partial_tj^{\varepsilon_2}\|_{H^s_{k\leq k_0}}^2\\&
 +(\varepsilon_1^2+\varepsilon_2^2)2^{4k_0}
 (1+\mathcal{E}^{\varepsilon_1}+
 \mathcal{E}^{\varepsilon_2})(\mathcal{D}^{\varepsilon_1}+\mathcal{D}^{\varepsilon_2})\Big)
+\tfrac{3}{4}\mathcal{D}^{\varepsilon_2,\varepsilon_1}_{>k_0}.
\end{align*}
 Applying the  Gr\"{o}nwall inequality, we obtain
\begin{align}\label{gron-diff}
 \sup_{t\geq 0}\mathcal{E}^{\varepsilon_2,\varepsilon_1}_{\leq k_0}(t)+\int_0^{+\infty}\mathcal{D}^{\varepsilon_2,\varepsilon_1}_{\leq k_0}(t)\d t
 \lesssim e^{C\int_0^{+\infty}(\mathcal{D}^{\varepsilon_1}+
 \mathcal{D}^{\varepsilon_2})\d t}\Big(\mathcal{E}^{\varepsilon_2,\varepsilon_1,in}_{\leq k_0}+\int_0^{+\infty}H(t)\d t\Big).
\end{align}

We now control the right-hand side terms. First, according to the uniform bound \eqref{uni-2}, we have 
\begin{align*}
\sup_{t\geq 0}(\mathcal{E}^{\varepsilon_1}(t)+
 \mathcal{E}^{\varepsilon_2} (t))\lesssim 1,\,\,
  \int_0^{+\infty}(\mathcal{D}^{\varepsilon_1}+
 \mathcal{D}^{\varepsilon_2})\d  t\lesssim 1.
\end{align*}
Next, due to the definition of $\mathcal{E}^{\varepsilon}_{>k_0},$ $\mathcal{D}^{\varepsilon}_{>k_0}$, $\mathcal{E}^{\varepsilon_2,\varepsilon_1}_{>k_0}$ and $\mathcal{D}^{\varepsilon_2,\varepsilon_1}_{>k_0}$, it is straightforward to see that
\begin{align}\label{e12}
\mathcal{E}^{\varepsilon_2,\varepsilon_1,in}_{\leq k_0}\leq \mathcal{E}^{\varepsilon_2,\varepsilon_1,in},\,\,\mathcal{E}^{\varepsilon_2,\varepsilon_1}_{>k_0}
  \lesssim \mathcal{E}^{\varepsilon_1}_{>k_0}+\mathcal{E}^{\varepsilon_2}_{>k_0},\,\,
\mathcal{D}^{\varepsilon_2,\varepsilon_1}_{>k_0}
  \lesssim \mathcal{D}^{\varepsilon_1}_{>k_0}+\mathcal{D}^{\varepsilon_2}_{>k_0}.
\end{align}
Hence, we deduce that 
\begin{align*}
&\int_0^{+\infty}(\mathcal{D}^{\varepsilon_1}+ \mathcal{D}^{\varepsilon_2})\mathcal{E}^{\varepsilon_2,\varepsilon_1}_{>k_0}\d t\lesssim \sup_{t\geq 0}(\mathcal{E}^{\varepsilon_1}_{>k_0}(t)+\mathcal{E}^{\varepsilon_2}_{>k_0}(t)),\\
 &\int_0^{+\infty}(\varepsilon_1^2+\varepsilon_2^2)2^{4k_0}
 (1+\mathcal{E}^{\varepsilon_1}+
 \mathcal{E}^{\varepsilon_2})(\mathcal{D}^{\varepsilon_1}+\mathcal{D}^{\varepsilon_1})\d t\lesssim (\varepsilon_1^2+\varepsilon_2^2)2^{4k_0},\\
 &\int_0^{+\infty}
\mathcal{D}^{\varepsilon_2,\varepsilon_1}_{>k_0}(t)\d t
\lesssim \int_0^{+\infty}
(\mathcal{D}^{\varepsilon_1}_{>k_0}(t)
+\mathcal{D}^{\varepsilon_2}_{>k_0}(t))\d t.
\end{align*}
Finally, by virtue of  Lemma \ref{hf}, we have \begin{align*}
\int_0^{+\infty}(\|\varepsilon_1^2\partial_t j^{\varepsilon_1}\|_{H^{s}_{\leq k_0}}^2+\|\varepsilon_2^2\partial_t j^{\varepsilon_2}\|_{H^{s}_{\leq k_0}}^2)\d t\lesssim \|\varepsilon_1 j^{\varepsilon_1,in}\|_{H^{s}}^2+\|\varepsilon_2 j^{\varepsilon_2,in}\|_{H^{s}}^2+(\varepsilon_1+\varepsilon_2)2^{2k_0}.
 \end{align*} 
Plugging all the above inequalities into  \eqref{gron-diff}, and using the fact  $0<\varepsilon_1,\varepsilon_2\leq 1$ and $k_0\geq -1$, complete the proof of the lemma.
\end{proof}
\subsection{Convergence to the limit system}
\
\newline
\indent
{\bf Proof of Theorem \ref{thm:Global-converg}.}
 Under the initial assumption
\begin{align}\label{assumption}
(u^{\varepsilon,in},\varepsilon j^{\varepsilon,in}, E^{\varepsilon,in},
  B^{\varepsilon,in}) \longrightarrow (u^{in},0, E^{in}, B^{in}) \,\,\, \textit{in}\,\,\, H^{s'}\times (H^{s})^3, \,\,\,{\text{as}}\,\,\, \varepsilon\rightarrow 0,
  \end{align}
after extracting an arbitrary sequence $\{\varepsilon_n\}$, we deduce from Lemma \ref{frequency} that, there exists a corresponding sequence $\{\omega_k\}$ of positive numbers satisfying
\begin{align*}
 \omega\in AF(\tfrac{1}{2}),\, \,\text{and}\,\,\lim\limits_{k\rightarrow +\infty}\omega_k=+\infty,
\end{align*}
such that for all $n$,
$(u^{\varepsilon_n,in}, \varepsilon_n j^{\varepsilon_n,in}, E^{\varepsilon_n,in},
  B^{\varepsilon_n,in}) \in H^{s'}(\omega)\times (H^{s}(\omega))^3,$
and
\begin{align*}
\sup_{n}(\|u^{\varepsilon_n,in}\|_{H^{s'}(\omega)}^2 + \| \varepsilon_nj^{\varepsilon_n,in}\|_{H^s(\omega)}^2
		  + \|E^{\varepsilon_n,in}\|_{H^{s}(\omega)}^2
		  + \|B^{\varepsilon_n,in}\|_{H^{s}(\omega)}^2)<\infty.
\end{align*}
Note that there also has a small initial assumption:
 \begin{align*}	
\|u^{\varepsilon,in}\|_{H^{s'}}^2 + \| \varepsilon j^{\varepsilon,in}\|_{H^s}^2
		  + \|E^{\varepsilon,in}\|_{H^{s}}^2
		  + \|B^{\varepsilon,in}\|_{H^{s}}^2
			\leq \kappa_0, \,\,{\text{for all}}\,\, \varepsilon>0,
	  \end{align*}
for some sufficiently small $\kappa_0$. We then conclude from Theorem \ref{thm:GNSM} that,  system \eqref{euqNSM} admits global solutions
$(u^{\varepsilon_n}, j^{\varepsilon_n}, E^{\varepsilon_n}, B^{\varepsilon_n})$ with the initial data  $(u^{\varepsilon_n,in},  j^{\varepsilon_n,in}, E^{\varepsilon_n,in},
  B^{\varepsilon_n,in})$ for all $n$. Furthermore, they are uniformly bounded with respect to $n$ as stated in \eqref{uni-2} and \eqref{uni-omega-ep}:
\begin{align}\label{gaoding}
	\sup_{n}\Big(\sup\limits_{t\geq 0}\mathcal{E}^{\varepsilon_n}(t)+
  \int_{0}^{+\infty} \mathcal{D}^{\varepsilon_n}(t)\d t\Big) \leq 2\kappa_0,
  \end{align}
\begin{align}\label{jiayou}
& \sup_{n}\Big(\sup\limits_{t\geq 0}\mathcal{E}^{\varepsilon_n}_{\omega}(t)+\int_{0}^{+\infty} \mathcal{D}^{\varepsilon_n}_{\omega}(t) \d t \Big)< \infty.
\end{align}

Since $\omega$ is increasing, similar to \eqref{w},  we have 
\begin{align*}
   \sup\limits_{t\geq 0}\mathcal{E}^{\varepsilon_n}_{>k_0}(t)+\int_0^{+\infty} \mathcal{D}^{\varepsilon_n}_{>k_0}(t) \d t
   \leq \frac{1}{\omega_k^2} \Big(\sup\limits_{t\geq 0}\mathcal{E}^{\varepsilon_n}_{\omega}(t)+\int_{0}^{+\infty} \mathcal{D}^{\varepsilon_n}_{\omega}(t) \d t\Big) .\end{align*}
Combining \eqref{jiayou} with $\lim\limits_{k\rightarrow+\infty}\omega_k=+\infty$, we deduce that,  $\forall \eta>0$, there exists a $k_0\geq -1$, such that
for all $n$,
\begin{align}\label{hign-diff}
\sup\limits_{t\geq 0}\mathcal{E}^{\varepsilon_n}_{>k_0}(t)+\int_0^\infty \mathcal{D}^{\varepsilon_n}_{>k_0}(t) \d t\leq \eta.
\end{align}
Fix $k_0$. Since $\lim\limits_{n\rightarrow +\infty}\varepsilon_n=0$,
there exists a $N_1(k_0,\eta)>0$, such that if $n>N_1$, then
\begin{align*}
  2^{4k_0}\varepsilon_n\leq \eta.
\end{align*}
Next, the initial assumption \eqref{assumption} implies that, there exists a $N_2>0$, such that if $n>N_2$, then
\begin{align*}
   \|u^{\varepsilon_n,in}-u^{in}\|_{H^{s'}}^2 + \|\varepsilon_n j^{\varepsilon_n,in}\|_{H^{s}}^2+
   \|E^{\varepsilon_n,in}-E^{in}\|_{H^{s}}^2
   +\|B^{\varepsilon_n,in}-B^{in}\|_{H^{s}}^2\leq \eta.
\end{align*}
This implies  if $n_1,n_2>N_2$, then
\begin{align*}
  \mathcal{E}^{\varepsilon_{n_2},\varepsilon_{n_1},in} + \|\varepsilon_{n_1} j^{\varepsilon_{n_1},in}\|_{H^{s}}^2+ \|\varepsilon_{n_2} j^{\varepsilon_{n_2},in}\|_{H^{s}}^2
   \leq 4\eta.
\end{align*}
According to Lemma \ref{low-diff}, we conclude that, if $n_1,n_2\geq N_0=\max\{N_1,N_2\}$, then
\begin{align*}
 \sup_{t\geq 0}\mathcal{E}^{n_2,n_1}_{\leq k_0}(t)+\int_0^{+\infty}\mathcal{D}^{n_2,n_1}_{\leq k_0}(t)\d t
 \leq &C\eta.
\end{align*}
Together with \eqref{dec-new} and \eqref{e12} and \eqref{hign-diff}, we eventually get that
\begin{align*}
 \sup_{t\geq 0}\mathcal{E}^{n_2,n_1}(t)+\int_0^{+\infty}\mathcal{D}^{n_2,n_1}(t)\d t
 \leq &C\eta.
 \end{align*}
 This implies for an arbitrarily sequence $\{\varepsilon_n\}$, $(u^{\varepsilon_n},\na u^{\varepsilon_n},j^{\varepsilon_n},E^{\varepsilon_n},B^{\varepsilon_n})$ is a Cauchy sequence in $\mathbb{X}$:
$$\mathbb{X}=L^\infty(\mathbb{R}^+;H^{s'})\times L^2(\mathbb{R}^+;{H}^{s'})\times L^2(\mathbb{R}^+;H^{s})\times (L^\infty(\mathbb{R}^+;H^{s})^2.$$

Let $(u,\na u, j,B,E)$ be the limit of $(u^\varepsilon,\na u^{\varepsilon}, j^\varepsilon,B^\varepsilon,E^\varepsilon)$ in $\mathbb{X}$.  By virtue of Lemma \ref{jb}, we have, as $\varepsilon \rightarrow 0$,
\begin{align*}
  &u^\varepsilon\cdot\nabla u^\varepsilon\rightarrow u\cdot\nabla u \,\, {\text{in}}\,\,L^2{(\mathbb{R}^+;H^{s'-1})},\,\,
    j^\varepsilon\times B^\varepsilon\rightarrow j\times B \,\, {\text{in}}\,\,L^2{(\mathbb{R}^+;H^{s'-1})},\\
  &u^\varepsilon\times B^\varepsilon\rightarrow u \times B \,\, {\text{in}}\,\,L^2{(\mathbb{R}^+;H^{s})},\,\,
   \varepsilon^2 j^\varepsilon\cdot\nabla j^\varepsilon\rightarrow 0 \,\, {\text{in}}\,\,L^2{(\mathbb{R}^+;H^{s-1})},\\
     &\varepsilon^2 u^\varepsilon\cdot \nabla j^\varepsilon\rightarrow 0 \,\, {\text{in}}\,\,L^2{(\mathbb{R}^+;H^{s-1})},\,\,
      \varepsilon^2 j^\varepsilon\cdot \nabla u^\varepsilon\rightarrow 0 \,\, {\text{in}}\,\,L^2{(\mathbb{R}^+;H^{s-1})}.
\end{align*}
Taking the Leray projection operator $\mathbb{P}$, we deduce from  \eqref{euqNSM}-(1) and \eqref{euqNSM}-(2)  that
\begin{align*}
&\partial_tu^\varepsilon+\mathbb{P}(u^\varepsilon\cdot \nabla u^\varepsilon)+\mathbb{P}(\varepsilon^2j^\varepsilon\cdot \nabla j^\varepsilon)-\mu\Delta u^\varepsilon=\mathbb{P}(j^\varepsilon\times B^\varepsilon),\\		&\varepsilon^2\partial_tj^\varepsilon+\mathbb{P}(\varepsilon^2u^\varepsilon\cdot \nabla j^\varepsilon+\varepsilon^2j^\varepsilon\cdot \nabla u^\varepsilon)-\varepsilon^2\mu\Delta j^\varepsilon+\frac{1}{\sigma}j^\varepsilon=cE^\varepsilon+
\mathbb{P}(u^\varepsilon\times B^\varepsilon), \end{align*}
with  $u^\varepsilon|_{t=0}=u^{\varepsilon,in}$, $j^\varepsilon|_{t=0}=j^{\varepsilon,in}$, $ \div u^\varepsilon=0$ and $ \div j^\varepsilon=0$.
Hence, passing to the limit, we conclude  that
\begin{align*}
\partial_tu+\mathbb{P}(u\cdot \nabla u)-\mu\Delta u= \mathbb{P}(j\times B), \,\, j=\sigma(cE+\mathbb{P}(u\times B)),
\end{align*}
with  $u|_{t=0}=u^{in}$, $ \div u=0$ and $\div j=0$.  The divergence free property then implies that there exists two functions $p(t,x)$ and $\tilde{p}(t,x)$, such that
\begin{align*}
	\partial_tu+u\cdot \nabla u-\mu\Delta u=-\nabla p+j\times B,\,\,
{\text{and}}\,\,
			j=\sigma(-\nabla \tilde{p}+cE+u\times B).
\end{align*}
Similarly, passing to the limit in \eqref{euqNSM}-(3) and \eqref{euqNSM}-(4), we get
\begin{align*}		\tfrac{1}{c}\partial_tE-\nabla \times B=-j, \,\,
			\tfrac{1}{c}\partial_tB+\nabla \times E=0,
\end{align*}
with $E|_{t=0}=E^{in}, B|_{t=0}=B^{in}$, $ \div E=0$ and $\div B=0$. Hence, the limit functions $(u,j,E,B)$ solve \eqref{NSMO}. This completes the whole proof of the theorem.
\section*{Acknowledgments} 
\label{sec:acknowledgments}

Guo is supported by ARC FT230100588.  Zhang is supported by the National Natural Science Foundation of China, 11801425. 

\end{document}